\documentclass[12pt,reqno]{article}
mm \textheight=9.1in

\usepackage{amssymb}
\usepackage{amsmath}
\usepackage{amsthm}
\usepackage{color}

\newcommand{\br}[3]{{$#1$}$\lower4pt\hbox{$\tp\atop\raise4pt \hbox{$\scriptscriptstyle{#2}$}$} ${$#3$}}
\newcommand{\tw}[3]{{$#1$}${\,\scriptscriptstyle {#2}}\atop\raise9pt\hbox{$\scriptstyle\tp$} ${$#3$}}
\newcommand{\ttps}[2]{{#1}\raise5pt\hbox{$\lower12pt\hbox{$\scriptstyle\tp$}\atop \lower0pt\hbox{$\tilde\;$}$}\raise4.5pt\hbox{${\scriptstyle{#2}}$}}
\newcommand{\st}[1]{\mbox{${\,\scriptscriptstyle {#1}}\atop\raise5.5pt\hbox{$*$}$}}

\newcommand{\rd}[1]{\mbox{${\,\scriptscriptstyle {#1}}\atop\raise5.5pt\hbox{$\bullet$}$}}
\newcommand{\rt}[1]{\otimes_\chi}
\newcommand{\lt}[1]{\mbox{${\,\scriptscriptstyle {#1}}\atop\raise5.5pt\hbox{$\ltimes$}$}}
\newcommand{\btr}{\raise1.2pt\hbox{$\scriptstyle\blacktriangleright$}\hspace{2pt}}
\newcommand{\btl}{\raise1.2pt\hbox{$\scriptstyle\blacktriangleleft$}\hspace{2pt}}

\newcommand{\lcr}{\raise1.0pt \hbox{${\scriptstyle\rightharpoonup}$}}
\newcommand{\rcr}{\raise1.0pt \hbox{${\scriptstyle\leftharpoonup}$}}

\newcommand{\ttp}{{\lower12pt\hbox{$\tp$}\atop \hbox{$\tilde\;$}}}

\newcommand{\id}{\mathrm{id}}

\newcommand{\Tc}{\mathcal{T}}

\newcommand{\Fin}{\mathrm{Fin}}

\newcommand{\Ac}{\mathcal{A}}

\newcommand{\Ic}{\mathcal{I}}

\newcommand{\Ru}{\mathcal{R}}

\newcommand{\Kc}{\mathcal{K}}
\newcommand{\Q}{\mathcal{Q}}
\renewcommand{\O}{\mathcal{O}}

\newcommand{\C}{\mathbb{C}}
\newcommand{\Sbb}{\mathbb{S}}

\newcommand{\Qbb}{\mathbb{Q}}
\newcommand{\Hbb}{\mathbb{H}}
\newcommand{\Z}{\mathbb{Z}}

\newcommand{\N}{\mathbb{N}}

\newcommand{\tp}{\otimes}

\newcommand{\zt}{\zeta}

\newcommand{\U}{U}

\newcommand{\ve}{\varepsilon}
\newcommand{\gm}{\gamma}
\newcommand{\dt}{\delta}

\newcommand{\op}{\oplus}
\newcommand{\la}{\lambda}
\newcommand{\tr}{\triangleright}
\newcommand{\tl}{\triangleleft}

\newcommand{\Char}{\mathrm{ch }}

\newcommand{\End}{\mathrm{End}}

\newcommand{\Span}{\mathrm{Span}}

\newcommand{\Hom}{\mathrm{Hom}}

\newcommand{\Tr}{\mathrm{Tr}}

\newcommand{\Rm}{\mathrm{R}}

\newcommand{\Ad}{\mathrm{Ad}}
\newcommand{\La}{\Lambda}

\newcommand{\g}{\mathfrak{g}}
\renewcommand{\b}{\mathfrak{b}}
\renewcommand{\k}{\mathfrak{k}}
\newcommand{\h}{\mathfrak{h}}

\newcommand{\s}{\mathfrak{s}}

\renewcommand{\o}{\mathfrak{o}}

\newcommand{\m}{\mathfrak{m}}
\newcommand{\eps}{\epsilon}

\newcommand{\nn}{\nonumber}
\newcommand{\p}{\mathfrak{p}}
\renewcommand{\l}{\mathfrak{l}}

\newcommand{\si}{\sigma}
\newcommand{\al}{\alpha}

\newcommand{\bt}{\beta}

\newcommand{\be}{\begin{eqnarray}}
\newcommand{\ee}{\end{eqnarray}}

\newtheorem{thm}{Theorem}[section]
\newtheorem{propn}[thm]{Proposition}
\newtheorem{lemma}[thm]{Lemma}
\newtheorem{corollary}[thm]{Corollary}

\newtheorem{remark}[thm]{Remark}

\newcount\prg

\newcommand{\parag}{\advance\prg by1 {\noindent\bf\thesection.\the\prg\hspace{6pt}}}

\begin{document}
\title{Pseudo-parabolic category \\ over quaternionic projective plane}
\author{
Gareth Jones${}^\dag$ and Andrey Mudrov${}^{\dag,\ddag,\sharp}$\footnote{Corresponding author}
\vspace{20pt}
\\
$\dag$
\small University of Leicester, \\
\small University Road,
LE1 7RH Leicester, UK
\\\\
$\ddag$
\small St. Petersburg Department
 of V.A.Steklov Institute of Mathematics
\\\small
 27 Fontanka, St. Petersburg, 191023, Russia
\\\\
\small
$\sharp$ Moscow Institute of Physics and Technology\\
\small
9 Institutskiy per., Dolgoprudny, Moscow Region,
141701, Russia
\\\vspace{5pt}
\small e-mail: gpj3@le.ac.uk, am405@le.ac.uk
}

\date{ }

\maketitle
\begin{abstract}
Quaternionic  projective plane $\mathbb{H} P^2$ is the next simplest conjugacy  class
of a complex symplectic group  with pseudo-Levi
stabilizer subgroup after  the sphere
$\mathbb{S}^4\simeq \mathbb{H} P^1$.
Its quantization gives rise to a module category $\mathcal{O}_t\bigl(\mathbb{H} P^2\bigr)$ over finite-dimensional representations of
the symplectic quantum group $U_q\bigl(\mathfrak{s}\mathfrak{p}(6)\bigr)$,
a full subcategory
in the BGG category $\mathcal{O}$. We  prove that $\mathcal{O}_t\bigl(\mathbb{H} P^2\bigr)$  is semi-simple
and equivalent to a category of
quantized equivariant vector bundles on $\mathbb{H} P^2$.
\end{abstract}
\\
{\small \underline{AMS classification codes}: 17B10, 17B37, 53D55.}
\\

\section{Introduction}
With every point $t$ of a maximal torus $T$ of a simple complex algebraic group $G$ one can associate a full subcategory $\O_t$
in the BGG category $\O$ of  the corresponding quantum group, $U_q(\g)$. This subcategory is additive and
stable under the tensor product with the category  $\Fin_q(\g)$ of finite-dimensional (quasi-classical) $U_q(\g)$-modules.
Its objects are submodules in tensor products of $V\in \Fin_q(\g)$ with a distinguished base module $M$ of highest weight $\la$ depending on $t$.
In generic situation, the locally finite part of $\End(M)$  is
an equivariant quantization $\Ac$ of the coordinate ring of $C_t=\Ad_G(t)$, the conjugacy class of $t$.
If  $\O_t$ is semi-simple, then its objects can be regarded as "representations" of quantum equivariant vector
bundles on $\Ad_G(t)$. According to the famous Serre-Swan theorem  \cite{S,Sw}, global sections of vector bundles on an affine variety
form finitely generated projective modules over its coordinate ring and {\em vice versa}. Finitely generated projective right  $\Ac$-modules
equivariant with respect to $U_q(\g)$ can be viewed as quantum equivariant vector bundles. They constitute a $\Fin_q(\g)$-module
category, $\Pr_q(\Ac,\g)$.

Equivalence of  $\Fin_q(\g)$-module categories $\O_t$ and  $\Pr_q(\Ac,\g)$ is established via functors
acting on objects as $\Pr_q(\Ac,\g)\ni \Gamma \mapsto \Gamma\tp_\Ac M \in \O_t$
and $\O_t\ni N\mapsto \Hom^\circ_\C(M,N)\in \Pr_q(\Ac,\g)$, where the circle designates the locally finite part with respect
to the $U_q(\g)$-action.
The module $M$ is absent in the classical picture as there is no faithful irreducible representation of a classical commutative coordinate ring.

Quantization of vector bundles is a natural extension of the deformation quantization programme for Poisson manifolds \cite{BFFLS}.
Vector bundles on non-commutative spaces are of interest in the K-theory \cite{Sheu}, non-commutative geometry \cite{C}, and non-commutative quantum field theory \cite{DN}.
There is one more area of their applications in connection with quantum symmetric pairs and universal K-matrices,
\cite{Let1,Kolb}.
If the class $C_t$ is a symmetric space, then there is a one-dimensional representation of $\Ac$ (a classical point on
quantized $C_t$).
It satisfies the reflection equation
\cite{KS} defining a coideal subalgebra $U_q(\k')\subset U_q(\g)$. Then
$\Ac$ can be realized as the subalgebra of $U_q(\k')$-invariants in the Hopf algebra of functions on the quantum group that is dual to $U_q(\g)$.
In the classical limit, $U_q(\k')$ turns into the centralizer  $U(\k')$ of a point $t'\in C_t$,
which is conjugate to the centralizer $U(\k)$ of the point $t$.

The representation theory
of $U_q(\k')$ is a challenge since $t'\not \in T$ (which is fixed for a quantum group) and the triangular decomposition  of $U_q(\g)$ is not compatible with that of $U_q(\k')$,
\cite{Let1,Let2}.
The category $\O_t$, if semi-simple,  plays the role of a bridge between $\Pr_q(\Ac,\g)$ and the category of finite-dimensional
$U_q(\k')$-modules via a chain of equivalences. This is discussed in details in \cite{M4} for quantum spheres.

Remark that an associated vector bundle  in the classical geometry is obtained via induction functor from a finite dimensional representation of
the stabilizer subgroup, which is a relatively simple thing. In the non-commutative world the picture is quite opposite. It is surprisingly easier to construct an apparently more complex vector bundle,
and arrive at the fiber via specialization at the (quantum) initial point, if any. This transition is demonstrated  for
projective spaces in \cite{M6}.

In the present paper we study the category $\O_t$ for $G=SP(6)$ and $t\in T$   one of $6$ points with
the stabilizer  $\simeq SP(4)\times SP(4)$ (they belong to two isomorphic conjugacy classes). In this case, $C_t$ is the quaternionic projective plane $\Hbb P^2$
which enters one of the two infinite series,  $\Hbb P^n$, of rank $1$ non-Hermitian symmetric conjugacy classes.
The other series comprises even spheres and has been studied in \cite{M4}.
However, the approach of \cite{M4} (as well of the last section in \cite{M5}) is special for $\Sbb^{2n}$ and cannot be extended any further.
The method we demonstrate here  on the  example of $\Hbb P^2$
works for any semi-simple conjugacy class comprising elements of finite order (e.g. symmetric conjugacy classes).
This method reduces the question of semi-simplicity of $\O_t$ to simplicity of $M$.

We prove that the module $M$ is irreducible in the case of $\Hbb P^2$  and explicitly construct an orthonormal basis with respect
to the contravariant form on it. Our approach is based on viewing $M$ as a module over $U_q(\l)\subset U_q(\g)$,
where  $\l \simeq \g\l(2)\op \s\p(2)$  is the maximal reductive Lie subalgebra in $\k$ such that $U(\l)$ is quantized as a Hopf
subalgebra in $U_q(\g)$. This is the content of Section 2.

In Section 3, we prove semi-simplicity of the category $\O_t$. It is an illustration of
the complete reducibility criterion for  tensor products of highest weight modules
based on  a contravariant form and Zhelobenko extremal cocycle \cite{M3,M5,Zh}.
We show that for every finite-dimensional quasi-classical $U_q(\g)$-module $V$ the tensor product
$V\tp M$ is completely reducible and its simple submodules are in a natural bijection with simple $\k$-submodules in
the classical $\g$-module $V$. This way we establish equivalence of $\O_t$ and $\Fin(\k)$
as Abelian categories.

In Section 4 we present a classical point on quantum  $\Hbb P^2$, i.e. a one-dimensional representation
of $\Ac$. It is a numerical solution of the reflection equation
that satisfies other relations of  quantized $\C[\Hbb P^2]$. Therein we describe  the
coideal subalgebra $U_q(\k')$.

In the last Section 5 we establish equivalence of the category $\O_t$ with the category $\Pr_q(\Ac,\g)$.
\subsection{Quantum group $U_q\bigl(\s\p(6)\bigr)$ and basic conventions}
In this paper, $\g=\s\p(6)$, $\k= \s\p(4)\op \s\p(2)$ and $\l=\g\l(2)\op \s\p(2)$.
There are inclusions $\g\supset \k\supset \l$ of Lie algebras,
which we describe by inclusions of their root bases as follows.
Both $\k$ and $\l$ are reductive subalgebras of maximal rank, i.e. they contain the Cartan subalgebra $\h$
of $\g$. Fix the inner product on $\h$ such that the long  root has length $2$.
All positive roots of $\g$ are expressed in an orthonormal basis of weights $\{\ve_i\}_{i=1}^{3}\in \h^*$ as $\Rm^+_\g=\{\ve_i\pm \ve_j\}_{i<j}\cup \{2\ve_i\}_{i=1}^{3}$.
Then $\al_i=\ve_i-\ve_{i+1}$, $i=1,2$, and $\al_3=2\ve_{3}$ form the basis of simple roots $\Pi_\g=\Pi$.
The basis of simple roots of $\k$ is $\Pi_\k=\{\al_{1}, 2\al_2+\al_3,\al_{3}\}$.
Note that the root $2\al_2+\al_3$ is not in $\Pi_\g$,
so $\k$ is not a Levi subalgebra in $\g$. On the contrary, $\l$ is the maximal subalgebra in $\k$ that is Levi in $\g$.
Its basis of simple roots is $\Pi_\l=\{\al_{1}, \al_{3}\}$.

For two elements $x$, $y$ of an associative algebra and a scalar $a$ we write
$[x,y]_a=xy-ayx$.
We say that $x$ and $y$ quasi-commute  if  $[x,y]_a=0$ for some $a\in \C$, and call the
algebra quasi-commutative if this holds for all pairs of its generators.

The quantum group $U_q(\g)$ is a $\C$-algebra with unit parameterized by a complex number $q$, which
is assumed not a root of unity,
\cite{ChP}. It is
generated by simple root vectors
$e_{i}$, $f_{i}$ (Chevalley generators),
and invertible Cartan generators $q^{h_{i}}$, $i=1,2,3$.
The elements $q^{\pm h_{i}}$ generate a commutative
subalgebra $U_q(\h)$ in $U_q(\g)$ isomorphic to the polynomial algebra on a torus.
They  obey the following commutation
relations with $e_{i}$, $f_{i}$:
$$
q^{h_i} e_{j}= q^{(\al_i,\al_j)} e_{j}q^{h_i}
\quad
q^{h_i} f_{j}= q^{-(\al_i,\al_j)} f_{j}q^{h_i}
\quad i,j=1,2,3.
$$
Furthermore,
$
[e_{i},f_{j}]=\delta_{ij} \frac{q^{h_{i}}-q^{-h_{i}}}{q-q^{-1}}
$ for all $i,j=1,2, 3$.
Non-adjacent positive Chevalley generators commute while  adjacent generators satisfy quantum Serre relations
$$
[e_i,[e_i,e_j]_q]_{\bar q}=0, \quad i,j=1,2, \>i\not =j
,
\quad [e_2,[e_2,[e_2,e_3]_{q^2}]]_{\bar q^2}=0,
\quad
[e_3,[e_3,e_2]_{\bar q^2}]_{q^2}=0,
$$
where $\bar q= q^{-1}$.
Similar relations hold for the negative Chevalley generators on replacement $f_i\to e_i$,
which extends to an involutive algebra
automorphism of $U_q(\g)$ with  $\si(q^{h_i})=q^{-h_i}$.

A comultiplication defined on the generators by
$$\Delta(f_i)= f_i\tp 1+q^{-h_i}\tp f_i,\quad\Delta(q^{\pm h_i})=q^{\pm h_i}\tp q^{\pm h_i},\quad\Delta(e_i)= e_i\tp q^{h_i}+1\tp e_i$$
makes $U_q(\g)$ a Hopf algebra. The assignment $q^{h_i}\mapsto 1$, $e_i\mapsto 0$, $f_i\mapsto 0$ extends to the counit
homomorphism $U_q(\g)\to \C$, then  antipode $\gamma$ acts on the generators by $q^{h_i}\mapsto q^{-h_i}$, $e_i\mapsto -e_iq^{-h_i}$, $f_i\mapsto -q^{h_i}f_i$.
It is an anti-algebra and anti-coalgebra automorphism of $U_q(\g)$.

The composition $\omega=\si\circ \gm$
 is an involutive automorphism of $U_q(\g)$ that preserves comultiplication and flips multiplication.

The Serre relations are homogeneous with respect to  the $U_q(\h)$-grading via its adjoint action on $U_q(\g)$.
They are determined by the corresponding weight, so we refer to a particular relation by its weight in what follows.

We  remind that a total ordering on the set of positive roots is called normal
if any $\al\in \Rm^+$ presentable as a sum $\al=\mu+\nu$ with $\mu,\nu\in \Rm^+$ lies between $\mu$ and $\nu$.
A reductive Lie subalgebra $\l\subset \g$ of maximal rank is called Levi if it has a basis $\Pi_\l$ of
simple roots which is a part of $\Pi$.
Then there is an ordering such that every element of $\Rm^+_{\g/\l}$ is preceding all elements of $\Rm_\l$.
In this paper, $\l$ designates the subalgebra $\g\l(2)\op \s\p(2)$ as agreed upon earlier.


With a normal ordering one can associate a system  $\{\tilde f_\al\}_{\al\in \Ru^+}\subset U_q(\g_-)$ of elements such that  ordered monomials
in  $\tilde f_\al$ form a PBW-like basis in $U_q(\g_-)$. In particular, the algebra $U_q(\g_-)$ is freely generated
 over $U_q(\l_-)$ by ordered monomials in $\tilde f_\al$ with $\al\in \Rm^+_{\g/\l}$.
In the classical limit, the elements $\tilde f_\al$ form a basis of root vectors in $\g_-$.
For a detailed construction of such a basis, the reader is referred to \cite{ChP}.

By $\La_\g$ we denote the root lattice of $\g$, i.e. a free Abelian group generated by  fundamental weights
relative to the fixed polarization of $\Rm$.
The semi-group of integral dominant weights is denoted by  $\La^+_\g$. All $U_q(\g)$-modules are assumed
diagonalizable over $U_q(\h)$.
A non-zero vector $v$ of a $U_q(\h)$-module $V$ is said to be of weight $\mu\in \h^*$ if
$q^{h_\al}v=q^{(\al,\mu)}v$ for all $\al\in \Pi^+$. Vectors of weight $\mu$ span a subspace in $V$ denoted by $V[\mu]$. The set of weights of $V$ is denoted by $\La(V)$.

Infinitesimal character of a $U_q(\h)$-module is defined as a formal sum
$
\sum_{\mu\in \La(V)} \dim V[\mu]_\mu e^{\mu}.
$
We write $\Char(V)\leqslant \Char(W)$ if $\dim V[\mu]\leqslant \dim W[\mu]$ for all $\mu$
and $\Char(V)< \Char(W)$ if this inequality is strict for some $\mu$.

By all $q$ we mean all not a root of unity; almost all $q$ stands for
all except for a finite set of values.

\section{Base module for $\Hbb P^2$}

In this section we study a $U_q(\g)$-module $M$ that generates the category of our interest.
 We prove its irreducibility and construct an orthonormal basis with respect
to a contravariant form on it.

Let $\kappa$ denote the half-sum of the positive roots of $\k$.
Regard roots (more generally, integaral weights) as characters of the maximal torus $T$ of the group $G$ (the torus has been fixed
 and its Lie algebra is  $\h$ participating in the construction of $U_q(\g)$).
Define base weight $\la\in \h^*$ as one featuring the property
$q^{2(\la,\al)}=\al(t)q^{2(\kappa-\rho,\al)}$, for all $\al\in \Pi_\g$,
where $\al(t)$ is the value of  root $\al$ on the initial point $t\in T$. It is the
eigenvalue of the operator $\Ad_t$ on the corresponding root space in $\g$ and, in particular,
$\al(t)=1$ once $\al \in \Pi_\k$.

Remark that  $\la$ is evaluated on  squared Cartan generators in the above equality.
Therefore   base weight is not uniquely determined by the point $t$
but up to a choice of sign in $\pm\sqrt{\al(t)}$
for each $\al \in \Pi_\g$. One can pick up any for $\la$, but we additionally
assume $ q^{(\la,\al)}=1$ for all $\al\in \Pi_\l=\Pi_\g\cap \Pi_\k$.
This is consistent with the conditions on $\la$ because $(\kappa,\al)=(\rho,\al)=1$ for such $\al$.
The rational for this will be explained later.

We fix the initial point $t$ by
$$
\ve_i(t)=
\left\{\begin{array}{ccc}
 -1, & i=1,2,\\
1,& i=3,
\end{array}
\right.
$$
so the base weight satisfies   $q^{2(\la,\ve_{3})}=q^{(\la,\ve_1-\ve_2)}=1$,
$q^{2(\la,\ve_1)}= q^{2(\la,\ve_2)}=-q^{-2}$.

Set $\dt=2\al_2+ \al_{3}$ and $f_{\dt}=f_2^2 f_{3}-(q^2+\bar q^2)f_2f_{3} f_2+ f_{3} f_2^2$.
It is easy to check that $f_\dt$ commutes with $f_{3}$ and $e_{3}$, cf. \cite{M7}.
Let $\hat M_\la$ denote the Verma module with highest weight $\la$ and define $M$ as the quotient of $\hat M_\la$ by its submodule
generated by singular vectors $f_11_\la$, $f_{3}1_\la$, and $f_\dt1_\la$.
It is isomorphic to $U_q(\g_-)/J$ as a $U_q(\g_-)$-module,  where $J\subset U_q(\g_-)$ is the left ideal generated by $f_1$, $f_{3}$, $f_\dt$.

The module $M$ supports quantization of the conjugacy class $\Hbb P^2$ in the sense that
its quantized coordinate ring  $\C_q[\Hbb P^2]$   can be represented as a
$U_q(\g)$-invariant subalgebra in $\End(M)$. Its explicit formulation in terms of generators and relations
is given in Section \ref{SecRE}.

As $\l$ is a Levi subalgebra in $\g$, its universal enveloping algebra is quantized to a Hopf subalgebra  $U_q(\l)\subset U_q(\g)$.
The module $M$ is a quotient of the parabolic Verma module of the same weight, by the submodule generated
by (the image of) $f_\dt 1_\la$.
It follows that $M$ is locally finite over $U_q(\l)$, \cite{M5}.
We will study $M$ regarding  it as a $U_q(\l)$-module; then our additional requirements $(\la,\al)=0$ for $\al\in \Pi_\l$  will keep us within the category of quasi-classical $U_q(\l)$-modules (deformations of classical $U(\l)$-modules. Note that  $M$ is not quasi-classical for entire $U_q(\g)$.

\begin{remark}
\label{ref_sphere}
Note that  $M$ contains a base module for the quantum $4$-sphere,
\cite{M7}. It is  generated by the highest vector, over the natural quantum subgroup $U_q\bigl(\s\p(4)\bigr)\simeq U_q\bigl(\s\o(5)\bigr)$
in $U_q\bigl(\s\p(6)\bigr)$. We have used this fact when constructing the singular vectors $f_\dt1_\la$ and $f_31_\la$ in the Verma module $\hat M_\la$. We will further
refer to results on $\Sbb^4$ in our study of  higher pseudo-parabolic modules over  $\Hbb P^2$
in Section \ref{SecPPM}.
\end{remark}

\subsection{$U_q(\l)$-module structure of $M$}
It turns out that highest vectors of finite dimensional $U_q(\l)$-submodules in $M$
belong to a subalgebra $\simeq U_q\bigl(\s\l(3)\bigr)\subset U_q(\g)$, which we describe next.

Set $\xi=\al_1+\al_2+\al_3$ and $\theta=\al_1+2\al_2+\al_3$ and define root vectors
$$ f_\xi=[[f_1,f_2]_{\bar q},f_{3}]_{\bar q^2},\quad f_\theta=[f_2,f_\xi]_{q}
,\quad
 e_\xi=[e_{3},[e_2,e_1]_q]_{q^2}, \quad e_\theta=[e_\xi,e_2]_{\bar q}.
$$
Remark that $e_{\phi}$ is proportional to $\si(f_\phi)$ for $\phi=\xi,\theta$.
The set $\{f_\xi, e_\xi, q^{\pm h_\xi}\}$    forms a quantum $\s\l(2)$-triple with
$
[e_\xi,f_\xi]=[2]_q[h_\xi]_q
$.

\begin{propn}
  The elements $e_2,f_2,q^{\pm h_2}, e_\xi,f_\xi, q^{\pm h_\xi}$ generate
  a subalgebra $U_q(\m)\subset U_q(\g)$ isomorphic to $U_q\bigl(\s\l(3)\bigr)$, with the
  set of simple roots $\{\al_2, \xi\}$.
\label{RAsl3}
\end{propn}
\begin{proof}
Observe that the set $\Rm_\m=\{\pm \al_2,\pm \xi, \pm \theta\}\subset \h^*$ is a root system of the $\s\l(3)$-type with
$$
(\xi,\xi)=2, \quad (\al_2,\al_2)=2, \quad (\xi,\al_2)=-1,
$$
so the commutation relations between the Cartan and simple root generators are correct.
Furthermore, it is straightforward to check that $[e_2,f_\xi]=0$ and $[e_\xi,f_2]=0$.
Finally, so long $f_\theta=[f_2,f_\xi]_q$, the Serre relations $[f_\theta,f_2]_{q}=0=[f_\xi,f_\theta]_{q}$ hold by
(\ref{Ap-Serre_step}) and  (\ref{Ap-xi-theta}). This also yields the Serre relations $[e_\theta,e_2]_{q}=0=[e_\xi,e_\theta]_{q}$ via the
involution $\si$.
\end{proof}
\noindent
Remark that the subalgebra
$U_q(\m)$
results from a Lusztig transformation of the subalgebra with the simple root basis $\al_1,\al_2$, see Appendix.

\begin{propn}
\label{l-singular}
  Vectors $\{f_2^k f_\theta^l 1_\la\}_{k,l\in \Z_+}\subset M$ are $U_q(\l)$-singular (killed by all $e_\al$ with $\al\in \Pi_\l$).
\end{propn}
\begin{proof}
Both $e_{1}$ and $e_{3}$ commute with $f_2$, so we check their interaction with $f_\theta$.
An easy calculation gives $[e_3,f_\theta]=0$ and $[e_1,f_\theta]=f_\dt q^{h_{1}}\in J U_q(\h)$.
Hence $f_2^kf_\theta^l 1_\la$ is annihilated by  $e_{1}$ and $e_3$, by virtue of (\ref{Ap-theta delta}).
\end{proof}
\begin{corollary}
  The vector $f_\theta$ belongs to the normalizer of the left ideal $J$.
\label{theta-normalizer}
\end{corollary}
\begin{proof}
  Indeed, $ f_\dt f_\theta\in J$ by (\ref{Ap-theta delta}). Furthermore, $f_\theta1_\la$ generates
  a finite-dimensional $U_q(\l)$-submodule in $M$. Since $(\la-\theta,\al_{i})=0$ for $i=1,3$, this submodule is trivial,
  hence $f_{1}f_\theta$ and $f_{3} f_\theta$ are in $J$.
\end{proof}
We denote by $B$ the set  $\{f_2^k f_\theta^l 1_\la\}_{k,l\in \Z_+}\subset M$. Our next objective is to show that
$B$ is a basis of the subspace of $U_q(\l_+)$-invariants in $M$.
Let  $L_{k,l}\subset M$ be the $U_q(\l)$-submodule generated by $f_2^k f_\theta^l 1_\la$
and set $L=\op_{k,l=0}^\infty L_{k,l}\subset M$.

Introduce notation $f_{ij}$ for $i\leqslant j$
by setting $f_{ii}=f_i$ and recursively $f_{i,j+1}=[f_{i,j},f_{i+1}]_a$, where $a=q^{(\al_i+\ldots +\al_j,\al_{j+1})}$.
Then Serre relations imply
\be
f_{1}f_2^{k}=
[k]_qf_2^{k-1} f_{12} + q^{-k} f_2^{k} f_{1},
\quad
f_{3}f_2^{k}=
-q^2[k]_{q^2}f_2^{k-1} f_{23} + q^{2k} f_2^{k} f_{3} \mod J
\label{aux_comm_rel}
\ee
since $f_\dt$ commutes with $f_2$ and $f_3$.
It will be also of use to write these formulas as
\be
f_2^{k-1} f_{12}=\frac{1}{[k]_q}f_{1}f_2^{k}
 \mod J,
\quad
f_2^{k-1} f_{23}=-\frac{1}{q^2[k]_{q^2}}
f_{3}f_2^{k}   \mod J
\label{aux_comm_rel1}.
\ee

\begin{lemma}
For all $k\geqslant 2$,
$
f_{1}f_{3}f_2^{k}=
[k]_{q^2} f_2^{k-2}\Bigl([k]_qf_2f_3 f_1f_2 -\frac{[k-1]_q[2]_q}{(1-\bar q^2)}f_\theta\Bigr) \mod J.
$
\label{fullness}
\end{lemma}
\begin{proof}
Pushing $f_{3}$ and then $f_{1}$ to the right in $f_{1}f_{3}f_2^{k}$
 we find it equal to
$$
-q^2f_{1} [k]_{q^2}f_2^{k-1} f_{23}\mod J=
-q^2[k]_{q^2} [k-1]_q f_2^{k-2} f_{12} f_{23}  -q^2q^{-k+1} [k]_{q^2}f_2^{k-1}f_{1} f_{23}\mod J,
$$
where we have used (\ref{aux_comm_rel}).
Expressing  $f_{12} f_{23}$ and $f_2f_{1} f_{23}$ on the right
through $f_2f_\xi$ and $f_\theta$ modulo $J$
we prove the lemma.
\end{proof}

\begin{propn}
  $L$ exhausts all of  $M$.
\end{propn}

\begin{proof}
  It is sufficient to check that the $U_q(\l)$-submodule $L$ is invariant under $U_q(\g_-)$ as it contains $1_\la$.
  That is so if and only it is $f_{2}$-invariant.

The elements $f_{ij}$ with $i<j$ quasi-commute with $f_1$ and $f_3$ unless $k=i-1$ or $k=j+1$.
Therefore
$$
f_2L\subset U_q(\l_-)f_{12}B+U_q(\l_-)f_{23}B+U_q(\l_-)f_{\xi}B+L.
$$
Notice that $f_{12}$ quasi-commutes with every power of  $f_2$ while $f_{23}$ quasi-commutes with it modulo $J$
because $f_\dt \in J$ commutes with $f_2$ and $f_3$.
Therefore we can further push them  to the right until they hit $f_\theta$-s and then apply (\ref{aux_comm_rel1}).
This way we prove $f_{12}B\subset L$   and
$f_{23}B  \subset L
$, with the help of Corollary \ref{theta-normalizer}.

Furthermore, push $f_\xi$ to the right  in the third term until it hits $f_\theta$-s, using   $[f_2,f_\xi]_q=f_\theta$.
Then for all $k,l\in \Z_+$ we get $f_\xi f_2^kf_\theta^l 1_\la=f_2^k f_\xi f_\theta^l 1_\la$ modulo $L$ because
$f_\theta$ quasi-commutes with $f_2$ by (\ref{Ap-Serre_step}). But  $ f_\xi f_\theta^l 1_\la \propto f_1f_3f_2 f_\theta^l 1_\la$
because $f_1$ and $f_3$ are in $J$ and kill $f_\theta^l 1_\la$ by Corollary \ref{theta-normalizer}.
Applying Lemma \ref{fullness} to $f_\xi f_2^kf_\theta^l1_\la \propto f_2^k f_1f_3f_2 f_\theta^l1_\la$  we prove $f_\xi B\subset L$.
Then  $f_2L\subset  U_q(\l_-)f_{\xi}B+L\subset L$, as required.
\end{proof}

If follows from Proposition \ref{RAsl3} that
\be
[e_{2},f_\theta^k]=[k]_qf_\xi f_\theta^{k-1} q^{-h_2}, \quad [e_\theta^k,f_\xi]=-q^{-(k-1)}[2]_q[k]_qe_\theta^{k-1}e_2q^{-h_\xi}.
\label{U(m)}
\ee
Setting $\la_i=(\al_i,\la)$ we get as a consequence that
\be
e_{2}f_2^lf_\theta^k 1_\la=[l]_q[\la_3-l-k]_qf_2^{l-1} f_\theta^k 1_\la +[k]_q q^{-\la_2}f_2^l f_\xi f_\theta^{k-1}1_\la.
\label{e2-action}
\ee
\begin{propn}
The module $M$ is irreducible.
\end{propn}
\begin{proof}
It is sufficient to check that none of the $U_q(\l)$-singular vectors $f_2^l f_\theta^k1_\la$
with $l+k>0$ is killed by $e_2$. For $k=0$  this is straightforward: $e_2^l f_2^l 1_\la=[l]_q \prod_{i=0}^{l-1}[\la_2-i]_q1_\la$.
This never turns zero because $q^{2\la_2}=q^{2(\la,\ve_2-\ve_3)}=q^{2(\la,\ve_2)}=-q^{-2}$.
For $k>0$, the operator $e_{1}e_{3}$ annihilates the first term in (\ref{e2-action}) and returns $f_2^{l+1} f_\theta^{k-1}1_\la$, up to a non-zero scalar multiplier, on the second.
Proceeding this way we obtain
$
(e_{1}e_{3}e_2)^kf_2^lf_\theta^k1_\la\varpropto f_2^{k+l}1_\la\not =0.
$
Therefore $f_2^lf_\theta^k1_\la\not =0$ and $f_2^lf_\theta^k1_\la\not \in \ker(e_2)$ unless $l+k=0$. Hence these vectors
are highest for different $U_q(\l)$-submodules in $M$ and none of them is singular for $U_q(\g)$.
\end{proof}
In summary, $M$ is isomorphic to the natural $U_q\bigl(\g\l(2)\bigr)-U_q\bigr(\s\l(2)\bigl)$-bimodule $\C_q[\End(\C^2)]$.
It is semi-simple and  multiplicity free.
In the classical limit, the subalgebra of $U(\l_+)$-invariants in $\C[\C^2\tp \C^2]\simeq \C[\End(\C^2)]$ is a polynomial algebra in two variables
generated by the principal minors of the coordinate matrix, see e.g. \cite{GW}.
In the quantum case, the space of $U_q(\l_+)$-invariants in $M$ is isomorphic to a polynomial algebra in quasi-commuting variables $f_2, f_\theta$.
\begin{corollary}
  The infinitesimal character of the base module $M$ equals $\prod_{\al\in \Rm^+\backslash\Rm^+_\k}(1-e^{-\al})^{-1} e^\la$.
\end{corollary}
\begin{proof}
Readily follows from an isomorphism $\C[\End(\C^2)]\simeq U(\g_-/\k_-)$ of $U(\l)$-modules.
\end{proof}
\subsection{Orthonormal basis in $M$}

A symmetric bilinear form $(.,.)$ on a $U_q(\g)$-module $V$ is called contravariant if $(x v, w)=(v,\omega (x)w)$ for
all $x\in U_q(\g)$ and all $v,w\in V$.
Recall that every highest weight module over a reductive quantum group has a unique contravariant form
with respect to the involution $\omega$ normalized to $1$ on the highest vector.
In this section we construct an orthonormal basis in $M$, with the help of the subalgebras $U_q(\l)$ and $U_q(\m)$.
It can be constructed as the Gelfand-Zeitlin basis in every $U_q(\l)$-submodule $L_{l,k}\subset M$, up to a common factor equal
 to the norm of the highest vector of $L_{l,k}$. Thus the problem essentially reduces to calculation of
those norms. That is done within a $U_q(\m)$-submodule in $M$ generated by $1_\la$ because the space of $U_q(\l_+)$-invariants is in that submodule.
\begin{propn}
 \label{recurrence for matrix elements}
 Set  $\la_\theta=(\la,\theta)$. Then the assignment $(l,k)\mapsto\tilde c_{l,k}=\langle 1_\la,e_\theta^k e_2^l f_2^l f_\theta^k1_\la\rangle$
  is a unique function $\Z_+^2\to \C$ satisfying
$$
\tilde c_{l,k}=-\tilde c_{l,k-1} [2]_q[k]_q^2 q^{-\la_\theta+l+1}+q^{-k} [l]_q[\la_2-l+1]_q\tilde c_{l-1,k}, \quad lk\not =0,
$$
$$
\hspace{-1.4in}\mbox{and }\quad\quad \quad \quad \quad \quad\quad \tilde c_{l,0}=[l]_q!\prod_{i=0}^{l-1}[\la_2-i]_q, \quad
\tilde c_{0,k}=[k]_q![2]_q^k\prod_{i=0}^{k-1}[\la_\theta-i]_q.
$$
\end{propn}
\begin{proof}
The boundary conditions easily follow from the basic relations of $U_q(\m)$. Uniqueness can be checked by an obvious
induction on $l+k$. To prove the recurrence relation permute $f_\theta^k$ and $f_2^l$, then
in the resulting matrix element $q^{-lk}\langle 1_\la,e_\theta^k e_2^l f_\theta^kf_2^l1_\la \rangle$ push one
copy of $e_2$ to the right:
$$
\tilde c_{l,k}=q^{-kl}\langle 1_\la, e_\theta^kf_\xi  e_2^{l-1}f_\theta^{k-1} f_2^l 1_\la\rangle  [k]_qq^{-\la_2+2l}
 +q^{-k}\langle 1_\la, e_\theta^k e_2^{l-1}f_2^{l-1} f_\theta^k  1_\la\rangle [l]_q[\la_2-l+1]_q
$$
$$
=-\tilde c_{l,k-1} [2]_q[k]_q^2 q^{-\la_\theta+l-1}+q^{-k} [l]_q[\la_2-l+1]_q\tilde c_{l-1,k}.
$$
This calculation is actually done in $U_q(\m)$. In particular, we used (\ref{U(m)}) and $[f_2,f_\theta]_{\bar q}=0$.
\end{proof}
\begin{propn}
\label{matrix elements}
The matrix element $c_{l,k}=\langle f_2^l f_\theta^k1_\la,f_2^l f_\theta^k1_\la\rangle$ equals $(-1)^{l+k} q^{k(k-5)+ lk+l(l-1)} \times q^{-l(\la,\al_2)}\tilde c_{l,k}$,
with
\be
\tilde c_{l,k}=[l]_q![k]_q![2]_q^k\prod_{i=0}^{l-1}[\la_2-i]_q\frac{\prod_{i=0}^{l+k-1}[\la_\theta-i]_q}{\prod_{i=0}^{l-1}[\la_\theta-i]_q}.
\label{mat_el_form}
\ee
\end{propn}
\begin{proof}
Let $\bar f_\theta\in U_q(\g_-)$ be the vector obtained from $f_\theta$ by the substitution $q^{-1}\to q$.
Using the formula (\ref{f_theta-bar_f_theta}),
replace $f_\theta$ with $q^{-2} \bar f_\theta$ in the left argument.
Then $c_{l,k}$ equals
$$
\langle f_2^l f_\theta^k 1_\la,f_2^lf_\theta^k 1_\la \rangle=(-1)^kq^{-2k}\langle f_2^l\bar f_\theta^k 1_\la,f_\theta^kf_2^l1_\la \rangle
=(-1)^{l} q^{-2k} \langle 1_\la,(q^{-h_\theta-4}e_\theta)^k(q^{-h_2}e_2)^lf_2^lf_\theta^k1_\la \rangle
$$
since $\omega(\bar f_\theta)=-q^{-h_\theta-4}e_\theta$.
One can express the right hand side through $\tilde c_{l,k}=\langle 1_\la,e_\theta^k e_2^l f_2^l f_\theta^k1_\la\rangle$ and check that $\tilde c_{l,k}$  defined by (\ref{mat_el_form}) satisfies the conditions of
Proposition (\ref{recurrence for matrix elements}).
\end{proof}
Note that $\la_\theta$ can be replaced with $\la_2$ because $q^{2\la_2}=-q^{2}=q^{2\la_\theta}$.
\begin{corollary}
  The system
$
y^{l,k}_{i,j}=\frac{1}{\sqrt{[2]_q^jd_{l,i}d_{l,j}c_{l,k}}}f_1^i f_3^j f_2^l f_\theta^k 1_\la$, where $l,k\in \Z_+$,
$i,j\leqslant l$, and $d_{l,m}=(-1)^mq^{-m(l-m+1)}[m]_q[l-m+1]_q $,
 is an orthonormal basis with
  respect to the contravariant form on $M$.
\end{corollary}
\section{Category $\O_t(\Hbb P^2)$}
\label{SecPPM}
While the base module $M$ supports a representation of  $\C_q[\Hbb P^2]$,
it generates a family of modules which may be regarded as "representations" of more general quantum vector bundles.
This interpretation is only possible if all such modules are completely reducible:  then they give rise
to projective modules over $\C_q[\Hbb P^2]$. They appear as submodules in tensor products $V\tp M$ (representing a trivial vector bundle),
for every  $V$ from the category  $\Fin_q(\g)$ of finite-dimensional quasi-classical $U_q(\g)$-modules.  Therefore the key issue is complete reducibility of tensor products $V\tp M$.
We solve this problem in the present section using a technique  developed in $\cite{M3,M5}$.
\subsection{Complete reducibility of tensor products}
\label{SecCRTP}
Suppose that $V$ and $Z$ are irreducible modules of highest weight. Each of them has a unique, upon a normalization,
nondegenerate contravariant  symmetric bilinear form, with respect to the involution $\omega\colon U_q(\g)\to U_q(\g)$.
Define a contravariant form on $V\tp Z$ as the product of the forms on the factors.
Then the module $V\tp Z$ is completely reducible if and only if the form on $V\tp Z$ is non-degenerate when
restricted to the span of singular vectors $(V\tp Z)^+$. Equivalently, if and only if every submodule of highest
weight in $V\tp Z$ is irreducible, \cite{M3}.

For practical calculations, it is convenient to deal with the pullback of the form under
an isomorphism of  $(V\tp Z)^+$ with a certain vector subspace in $V$ (alternatively, in $Z$)
which is defined as follows. Let $I^-_Z\subset U_q(\g_-)$ be the left
ideal annihilating a vector $1_\zt\in Z$ of highest weight $\zt$, and $I^+_Z=\si(I^-_Z)$ a left ideal in $U_q(\g_+)$.
Denote by $V^+_Z\subset V$ the kernel of $I^+_Z$, i.e. the subspace of vectors killed by $I^+_Z$.
Since $I^+_Z$ is $U_q(\h)$-invariant, $V^+_\Z$ is  $U_q(\h)$-invariant too.
There is a linear isomorphism between $V^+_Z$ and $(V\tp Z)^+$
assigning a singular vector $u=v\tp 1_\zt+\ldots $ to any weight vector $v\in V^+_Z$. Here we suppressed
the terms whose tensor $Z$-factors have lower weights than  $\zt$. Note that the isomorphism
$V^+_Z \to (V\tp Z)^+$ is "almost" $U_q(\h)$-equivariant: it shifts weights by $\zt$.

The pullback of the contravariant form under the map $V^+_Z\to (V\tp Z)^+$
 can be expressed through the contravariant form $\langle -, -\rangle$ on $V$
as $\langle \theta(v), w\rangle$, for a certain linear map $\theta$ on $V^+_Z$ with values in its dual space.
We call it extremal twist defined by $Z$.
In this paper, the contravariant form on $V$ is always non-degenerate when restricted to $V^+_Z$,
so we can write $\theta \in \End(V^+_Z)$.
This operator is related with the extremal projector $p_\g$, which is an
element of a certain extension $\hat U_q(\g)$ of $U_q(\g)$,  \cite{KT}.
It is constructed as follows.

A normal order on $\Rm^+$ defines an
embedding $\iota_\al\colon U_q\bigl(\s\l(2)\bigr)\to U_q(\g)$ for each $\al \in \Rm^+$, \cite{ChP}.
It acts by the assignment
$$
q\to q_\al=q^{\frac{(\al,\al)}{2}}, \quad e\to \tilde e_\al, \quad f_\al \to \tilde  f_\al, \quad q^h\to q^{h_\al},
$$
where $e,f$ and $q^h$ are the standard generators of $U_q\bigl(\s\l(2)\bigr)$ and the twiddled elements
are root vectors constructed via Lusztig automorphisms, \cite{ChP}.
For $\psi \in \h^*$, set $p_\g(\psi)$,  to be an ordered product
\be
p_\g(\psi) =\prod_{\al\in \Rm^+}^<p_\al\bigl((\psi+\rho, \al^\vee )\bigr),
\label{shifted_proj}
\ee
where $p_\al(z)$ is the image of
\be
\label{translationed_proj}
p(z)=\sum_{k=0}^\infty  f^k e^k \frac{(-1)^{k}q^{k(z-1)}}{[k]_q!\prod_{i=1}^{k}[h+z+i]_q}
\in \hat U_q\bigl(\s\l(2)\bigr), \quad z\in \C,
\ee
under $\iota_\al$.
For generic $\psi$, the operator   $p_\g(\psi)$ is well defined and  invertible on every finite-dimensional $U_q(\g)$-module. The specialization $p_\g=p_\g(0)$ is an idempotent satisfying
$e_\al p_\g=0=p_\g f_\al$ for all $\al\in \Pi$. This idempotent is called extremal projector.

The element $p_\g(\psi)$ gives rise to a rational trigonometric operator function of weight in every weight $U_q(\g)$-module that is
locally nilpotent over $U_q(\g_-)$.
\begin{thm}[\cite{M5}]
  \label{com_red_crit}
  Suppose that the map $p_\g(0)\colon V^+_Z\tp 1_{\zt}\to (V\tp Z)^+$ is well defined. Then
  $p_\g(\zt)$ is well defined as an operator on $ V_{Z}^+$. If $p_\g(\zt)$  invertible, then $\theta = p^{-1}(\zt)$.
\end{thm}
In the case of our concern, $p_\g=p_\g(0)$ is well defined, cf. Proposition \ref{ext_proj_reg} below.
However, the operator $p_\g(\zt)$  may have poles as a function of $\zt$.
The above theorem implies that such poles are removable.
In the special case of the fundamental module $V=\C^6$ all weights in  $V^+_Z$ are multiplicity free. Then
$\det(\theta)\propto \prod_{\al\in \Rm^+}\prod_{\mu\in \La(V)}\theta^\al_\mu$ up to a non-zero factor,  with
\be
\theta^\al_\mu=\prod_{k=1}^{l_{\mu,\al}}\frac{[(\zt+\rho+\mu,\al^\vee)+k]_{q_\al}}{[(\zt+\rho,\al^\vee)-k]_{q_\al}}.
\label{theta igenvalues}
\ee
Here $l_{\mu,\al}$ is the maximal integer $k$ such that  $\tilde e_{\al}^kV^+[\mu]\not =\{0\}$ for
$\tilde e_\al=\iota_\al(e)$. We compute $\theta$ in
the next section.

 \subsection{Extremal twist and extremal projector}
 \label{SecETEP}
In this section we calculate the determinant of the extremal twist defined by the base module $M$ using its relation to extremal projector
 and show that it does not vanish at all $q$.

Denote  simple positive roots of the Lie subalgebra $\k\subset \g$ by $\bt_1=\al_1, \bt_2=\dt,\bt_3=\al_3$.
The corresponding fundamental weights  of  $\k$ are
$
\mu_1=\ve_1$, $\mu_2=\ve_1+\ve_2$, $\mu_{3}=\ve_{3}$.
Pick up an integral dominant (with respect to $\k$)  weight $\xi=\sum_{s=1}^{3}i_s\mu_s$ with  $\vec i= (i_s)_{s=1}^{3} \in \Z^{3}_+$
and set $\zt=\xi+\la$.
The Verma module $\hat M_{\zt}$ of highest weight $\zt$ and highest vector $1_\zt$ has
singular vectors
$\bar F_s^{i_s+1}1_{\zt}$, where $\bar F_s=f_s$, $s=1, 3$,  and
$$
\bar F_2=\bar q^2\Bigl( f_2^2 f_{3} \frac{[h_2-1]_p}{[h_2+1]_q}-f_2 f_{3} f_2[2]_q\frac{[h_2-1]_q}{[h_2]_q}+ f_{3} f_2^2\Bigr)\in \hat U_q(\b_-).
$$
That is straightforward for $\bar F_1^{i_1+1}1_{\zt}$ and for $\bar F_3^{i_3+1}1_{\zt}$ and follows from
\cite{M4}, Proposition 2.7, since $1_\zt$ generates a Verma  submodule over
the quantum subgroup $U_q(\s\p(4))$, cf. Remark \ref{ref_sphere}.

Denote by $\tilde M_{\vec i}$ the quotient of $\hat M_{\zt}$ by the submodule generated by $\{\bar F_s^{i_s+1}1_{\zt}\}_{s=1}^{3}$.
    The projection $\hat M_\zt\to \tilde M_{\vec i}$ factors through a parabolic Verma module relative to $U_q(\l)$: it is
     the quotient of $\hat M_\zt$ by the submodule generated by $\{f_s^{i_s+1}1_{\zt}\}_{s=1,3}$.
Therefore $\tilde M_{\vec i}$ is locally finite over $U_q(\l)$, \cite{M4}.
We use the same notation $1_\zt$ for the highest vector in $\tilde M_{\vec i}$.

Denote by $F_s^{i_s+1}\in U_q(\g_-)$ the Shapovalov elements, i.e. the images of singular vectors
 $\bar F_s^{i_s+1}1_{\zt}$ under the natural isomorphisms
 $
 U_q(\g_-)\simeq \hat M_\zt
 $,
 and set $E_s^{i_s+1}=\si(F_s^{i_s+1})\in U_q(\g_+)$.
Note with care that, contrary to  $\bar F_2^{i_2+1}$,  the elements  $F_2^{i_2+1}$  are not powers of  $F_2$.

 Let $\tilde I^-_{\vec i}\subset U_q(\g_-)$ denote the left ideal annihilating the highest vector in $\tilde M_{\vec i}$.
 and put $\tilde I^+_{\vec i}=\si(\tilde I^-_{\vec i})\subset U_q(\g_+)$.
 These ideals are generated by $\{F^{i_s+1}_s\}_{s=1}^3$ and $\{E^{i_s+1}_s\}_{s=1}^3$, respectively,
For $i=1,3$, these generators are simply  powers of simple root vectors.

From now to the end of the section we fix $V=\C^{6}$, the smallest fundamental module of $U_q(\g)$.
Up to non-zero scalar factors, the action of $U_q(\g_+)$ on $V$ is described by a graph
\be
v_{-1}\stackrel{e_1}{\longrightarrow} v_{-2}\stackrel{e_2}{\longrightarrow} v_{-3}\stackrel{e_3}{\longrightarrow} v_{3}\stackrel{e_2}{\longrightarrow} v_{2}\stackrel{e_1}{\longrightarrow} v_{1}
\label{graph}
\ee
where the vectors $v_{\pm i}$ of weights $\pm \ve_i$, $i=1,2,3$, form
an orthonormal basis with respect to the contravariant form.
The diagram for the $U_q(\g_-)$-action is obtained by reversing the arrows
in (\ref{graph}).
We find from the diagram that $\ker(E_s)$ equals
\be
V\ominus\Span\{v_{-1},v_{2}\},\> s=1,\quad
V\ominus \Span\{v_{-2}\}, \> s=2,\quad
V\ominus \Span\{v_{-{3}}\}, \> s=3.
\label{ker_E}
\ee
Furthermore,  $\ker(E_s^i)$ is entire $V$ if $i>1$.
 We denote by $\tilde V^+_{\vec i}=\cap_{s=1}^{3} \ker(E^{i_s+1}_s)$ the
 kernel of the left ideal $\tilde I^+_{\vec i}$ in $V$.

\begin{propn}
\label{ext_proj_reg}
  The extremal projector $p_\g\colon \tilde V^+_{\vec i}\tp 1_\zt\to (V\tp \tilde M_{\vec i})^+$  is well defined.
\end{propn}
\begin{proof}
It is argued in \cite{M5} that the factors $p_\al(z)$ in (\ref{shifted_proj}) for $\al\in \Rm^+_\l$ are regular on $\tilde V^+_{\vec i}\tp 1_{\zt}$
at $z=(\rho,\al^\vee)$ because all weights in $\tilde V^+_{\vec i}\tp 1_{\zt}$ are $\k$- and therefore $\l$-dominant.

Suppose that $\al \in \Rm^+_\g\backslash\Rm^+_\l$ and evaluate
the denominators in $p_\al(z)$ at $z=(\rho, \al^\vee)$  on a tensor of weight $\eta=\mu+\zt$,  $\mu\in \La(\tilde V^+_{\vec i})$.
They contain
$[z+(\eta,\al^\vee)+k]_{q_\al}$ with  $k\in \N$. For $\al\in \Rm^+_\g\backslash\Rm^+_\k$,
such a  factor is  proportional to  $q^{x}+q^{-x}$ for some $x\in \Qbb$ and does not vanish because $q$ is not
a root of unity. Therefore
all  factors $p_\al(t)$ for such $\al$ are regular at $z=(\rho,\al^\vee)$. Moreover, the extremal projector of
the subalgebra $U_q(\g^{\al_2})$ is well defined on $V\tp 1_{\zt}$ taking it to $\ker  e_2$.

Now suppose that $\al\in \Rm^+_\k\backslash\Rm^+_\l$. With $\xi=0$ (i.e. $\zt=\la$), the factor $[(\eta+\rho,\al^\vee)+k]_q$ entering
$p_\al(t)$ is equal, upon evaluation of $h_\al$ on the subspace of weight $\eta=\mu+\zt$ in $\tilde V^+_{\vec i}\tp 1_{\zt}$, to
$$
[(\mu,\al^\vee)+2+k]_{q^2},\quad [(\mu,\al^\vee)+1+k]_{q^2}, \quad [(\mu,\al^\vee)+3+k]_{q},
$$
for $\al=2\ve_1,2\ve_2,\ve_1+\ve_2$, respectively. They are not zero since  $k>0$ and $(\mu,\al^\vee)\in \{-1,0,1\}$
for $\mu\in \La(\tilde V^+_{\vec i})$.
That is {\em a fortiori} true when $\xi\not =0$ because $(\xi,\al^\vee)\in \Z_+$. Therefore such $p_\al(t)$ are also
regular on $\tilde V^+_{\vec i}\tp 1_\zt$ at $t=(\rho,\al^\vee)$.

Thus all root factors in $p_\g(\psi)$ are regular on $\tilde V^+_{\vec i}\tp 1_{\zt}$ at $\psi=0$, so $p_\g(0)$ is independent of normal ordering.
For a simple root $\al$ choose an order with $\al$ on the left. Then $e_\al p_\g(0)=0$ on $\tilde V^+_{\vec i}\tp 1_{\zt}$.
We already saw  that for $\al=\al_2$; for $\al=\al_1,\al_3$
this is true because $V\tp M_{\vec i}$ is locally finite over $U_q(\l)$ and all weights in $\tilde V^+_{\vec i}\tp 1_{\zt}$ are dominant with respect to $\l$, cf. \cite{M5}.
This completes the proof.
\end{proof}
\noindent
Thus the first condition of Theorem \ref{com_red_crit} is satisfied. The second condition will be secured by the following calculation.
\begin{propn}
  \label{complete_red}
  For all $\xi= \sum_{s=1}^{3}i_s\mu_s$ with $\vec i\in \Z^3_+$, the operator $p_\g(\xi+\la)$  is invertible on $\tilde V^+_{\vec i}$.
\end{propn}
\begin{proof}
Let us calculate $\theta^\al_\mu$, which are inverse eigenvalues
of the root factors constituting $p_\g(\zt)$.
From (\ref{graph}) we conclude that all integers $l_{\mu,\al}$ in (\ref{theta igenvalues}) are at most $1$.
Put $\zt=\la+\xi$, then  (\ref{theta igenvalues}) reduces to  $\theta^\al_\mu=1$ for $l_{\al,\mu}=0$ and to $\theta^\al_\mu=\frac{[(\zt+\rho+\mu,\al^\vee)+1]_{q_\al}}{[(\zt+\rho,\al^\vee)-1]_{q_\al}}$
for $l_{\al,\mu}=1$.
Observe
that
$$
\theta^{\ve_1-\ve_3}_{-\ve_1},\quad \theta^{\ve_1-\ve_3}_{\ve_3},
\quad
\theta^{\ve_2-\ve_3}_{-\ve_2}, \quad \theta^{\ve_2-\ve_3}_{\ve_3},
\quad
\theta^{\ve_2+\ve_3}_{-\ve_2}, \quad \theta^{\ve_2+\ve_3}_{-\ve_3},\quad
\theta^{\ve_1+\ve_3}_{-\ve_1},\quad\theta^{\ve_1+\ve_3}_{-\ve_3}.
$$
are all of the form $\frac{\{m_1\}_q}{\{m_2\}_q}$ for some integers $m_1,m_2$, where
 $\{x\}_{q}=\frac{q^{x}+q^{-x}}{q+q^{-1}}$. They cannot turn zero as $q$ is not a root of unity.
The remaining non-trivial factors  $\theta^\al_\mu$ are
$$
\theta^{2\ve_1}_{-\ve_1}=\frac{[i_1+i_2+2]_{q^2}}{[i_1+i_2+1]_{q^2}}, \quad
\theta^{2\ve_2}_{-\ve_2}=\frac{[i_2+1]_{q^2}}{[i_2]_{q^2}},
\quad
\theta^{2\ve_3}_{-\ve_3}= \frac{[i_{3}+1]_{q^2}}{[i_{3}]_{q^2}},
$$
$$
\theta^{\ve_1-\ve_2}_{-\ve_1}=
\frac{[i_1+1]_{q}}{[i_1]_{q}}= \theta^{\ve_1-\ve_2}_{\ve_2},\quad
\theta^{\ve_1+\ve_2}_{-\ve_1}=
\frac{[i_1+2i_2+3]_{q}}{[i_1+2i_2+2]_{q}}=\theta^{\ve_1+\ve_2}_{-\ve_3}.
\quad
$$
Observe that the denominator in
$\theta^\al_\mu$ may vanish only
for $\al\in  \Pi_\k$. That happens if $i_s=0$ for some $s=1,2,3$. However, such $\mu$ do not belong to $\La(\tilde V^{+}_{\vec i})$,
as seen from  (\ref{ker_E}).
Since $q$ is not a root of unity, all $\theta^\al_\mu$ never turn zero. Therefore,  $p_\g(\zt)$ is invertible, and $\theta=p_\g(\zt)^{-1}$.
\end{proof}
In the next section we shall see that the kernels $\tilde V_{\vec i}^+$   parameterise irreducible decompositions
in a pseudo-parabolic category associated with $\Hbb P^2$.
\subsection{Pseudo-parabolic category  $\O_t(\Hbb P^2)$ and its structure}
In this section we define the pseudo-parabolic category over $\Hbb P^2$, prove its semi-simplicity and describe simple objects,
based on the results of the previous section.

Denote by $\O_t(\Hbb P^2)$ a full subcategory in the category $\O$ whose objects are submodules in $W\tp M$,
where $W\in \Fin_q(\g)$
is a quasi-classical finite-dimensional module over $U_q(\g)$. It is a module category over $\Fin_q(\g)$
because for every submodule $N\subset W\tp M$ and $U\in \Fin_q(\g)$, the module $U\tp N$ is in $U\tp W\tp M$.

We denote by  $\mathrm{Fin}(\k)$ the tensor category of finite-dimensional $\k$-modules. It is a module category
over $\mathrm{Fin}(\g)$ via the restriction functor.

Let $M_{\vec i}$ denote the irreducible quotient of $\tilde M_{\vec i}$ (we will later prove that they coincide at almost all $q$).
We call it pseudo-parabolic Verma module of the corresponding highest weight.

We define $V_{\vec i}^+$ as the kernel of the left ideal $I^+_{\vec i}=\si(I^-_{\vec i})$,
where $I^-_{\vec i}$ is the annihilator of the highest vector in $M_{\vec i}$. Obviously $V^+_{\vec i} \subseteq \tilde V^+_{\vec i}$
because  $\tilde I_{\vec i}^+\subseteq I_{\vec i}^+$.
The subspace $V^+_{\vec i}$ is isomorphic to the span of singular vectors in $V\tp M_{\vec i}$, in compliance with discussion of
Section \ref{SecCRTP}.
In principle, $\tilde V^+_{\vec i}$ might  be bigger than $V^+_{\vec i}$ but we shall see that they coincide for almost all $q$
(for all if $\dim V=6$.

From now until Corollary \ref{multiplicities} we assume that $V=\C^6$.
Let $X_{\vec i}\in \Fin(\k)$,  with
$\vec i\in \Z_+^3$,  denote the finite-dimensional $\k$-module of highest weight $\xi=\sum_{s=1}^{3}i_s\mu_s$.
For each $\vec i\in \Z_+^3$, introduce a set of triples $\tilde I(\vec i)\subset \Z_+^3$:
\be
\tilde I(\vec i)=\bigl\{(i_1\pm 1,i_2, i_{3}), \> (i_1, i_2, i_{3}\pm 1), \> (i_{1}\pm 1,i_{2}\mp 1,i_3)\bigr\},
\label{checked summands}
\ee
where  those with negative coordinates are excluded.
 Elements of $\tilde I(\vec i)$ parameterize irreducible  $\k$-submodules in  $V\tp X_{\vec i}$:
 their components are coordinates of highest weights in the basis of fundamental weights $\{\mu_s\}_{s=1}^3$.

Let $\Fin(\g\downarrow\k)$ denote the subcategory of $\k$-modules that are submodules in  modules from $\Fin(\g)$.
\begin{propn}
    $\Fin(\g\downarrow\k)\simeq\Fin(\k)$.
  \label{all Verma in PPO}
\end{propn}
\begin{proof}
Since $\Fin(\g)$ is generated by $V$ as a tensor category, it is sufficient to prove that
for each $\vec i\in \Z_+^3$ the $\k$-module $X_{\vec i}$ is in some tensor power of $V$.
  We do it by  induction on $|\vec i|=i_1+i_2+ i_{3}$.
 For $|\vec i|=0$, $X_{\vec i}$ is the trivial module  $\C$, which is in $\Fin(\g\downarrow\k)$.
Suppose that the statement is proved for all $X_{\vec i}$ with  $|\vec i|=m\geqslant 0$.
Fix an index $\vec i$ with $|\vec i|=m+1$ and let $\ell$ be the minimal $s\in \{1,2,3\}$ such that $i_s>0$.
We will separately  consider  two cases depending on the value of $\ell$.

For $\ell=1,3$  we define $\vec j^\ell\in \Z_+^3$ by setting $j^{\ell}_s=i_s-\dt_{s\ell}$.
 Then $\vec i \in \tilde I(\vec j^{\ell})$,  as follows from (\ref{checked summands}).
Since $|\vec j^{\ell}|=m$ by assumption,   $X_{\vec i}$ is in $\Fin(\g\downarrow\k)$.

In the case of  $\ell=2$ we consider  a pair of vectors $\vec j, \vec k \in \Z^{3}_+$ via
$j_s= i_s-\dt_{2 s}$ and $k_s= i_s+\dt_{1s}-\dt_{2 s}$ for $s=1,2,3$.
Since  $|\vec j|=m$, the module $X_{\vec j}$ is in $\Fin(\g\downarrow\k)$ by the induction assumption.
Now observe from (\ref{checked summands}) that $\vec k\in \tilde I(\vec j)$ and $\vec i\in \tilde I(\vec k)$.
Therefore  $ X_{\vec k}$ and $ X_{\vec i}$ are in $\Fin(\g\downarrow\k)$.
This completes the proof.
\end{proof}

Let us  denote by
 $f_{\beta_s},e_{\beta_s}\in \k$, $s=1,2,3$, its negative and positive (classical) Chevalley generators.
\begin{lemma}
For all $i\in \Z_+$, there are isomorphisms  $\ker (F_s^i)\simeq \ker (f_{\beta_s}^i)$ and $\ker (E_s^i)\simeq \ker (e_{\beta_s}^i)$,
 $s=1,2,3$, in $V$.
\label{class-quant}
\end{lemma}
\begin{proof}
Notice that the case of $i>1$ is easy because the kernels coincide with the whole $V$.
The case $i=1$ is an elementary calculation based on the diagram (\ref{graph}).
\end{proof}
\noindent

\begin{corollary}
  \label{classical_ext_space}
 The vector space  $\tilde V^+_{\vec i}$ is isomorphic to  $(V\tp X_{\vec i})^{\k_+}$.
 \end{corollary}
\begin{proof}
First of all observe that $\cap_{s=1}^3\ker (e_{\bt_s}^{i_s+1}) \simeq \cap_{s=1}^3\ker (E_{\bt_s}^{i_s+1})$ because all weights
in $V$ are multiplicity free.
Then the statement is due to the isomorphism $\tilde V_{\vec i}^+\simeq \cap_{s=1}^3\ker (e_{\bt_s}^{i_s+1})$ because the
right-hand side is in bijection with the span of singular vectors in the $\k$-module $V\tp X_{\vec i}$.
\end{proof}

\begin{propn}
The category $\O_t(\Hbb P^2)$ is semi-simple,  and its simple objects are $U_q(\g)$-modules of  highest weights $\la+\xi$, $\xi\in \La^+_\k$.
\label{VMi-comp-red}
\end{propn}
\begin{proof}
Since $V^+_{\vec i}\subseteq \tilde V^+_{\vec i}$ and $M_{\vec i}$ is a quotient of $\tilde M_{\vec i}$,
the extremal projector $p_\g\colon V^+_{\vec i}\tp 1_\zt\to (V\tp M_{\vec i})^+$ is well defined,
  by Proposition \ref{ext_proj_reg}. The operator $p_\g(\zt)$ is invertible on  $V^+_{\vec i}$
  by Proposition \ref{complete_red}.
  Then the tensor product $V\tp M_{\vec i}$ is completely reducible, for each  $\vec i\in \Z^3_+$, thanks to Theorem \ref{com_red_crit}.
  The highest weights of irreducible submodules are from $\la+\La(\tilde V^+_{\vec i}\tp 1_{\vec i})\subset \la+\La^+_\k$.

  Now observe that every module from $\Fin_q(\g)$ can be realized as a submodule in a tensor power of $V$.
  Applying induction on $m\in \Z_+$ such that $M_{\vec i}\subset V^{\tp m}\tp M$ (for $m=0$ the statement is obvious)
  we prove that all modules from $\O_t(\Hbb P^2)$  are completely reducible and the weights of irreducible components
  are as stated.
\end{proof}

Since $M_{\vec i}$ is a quotient of $\tilde M_{\vec i}$,
singular vectors in $V\tp M_{\vec i}$ may have only weights $\sum_{s=1}^{3} j_s\mu_s+\la$
with $\vec j\in \tilde I(\vec i)$, by Proposition \ref{VMi-comp-red}.
 Let $I(\vec i)\subseteq \tilde I(\vec i)$  denote the subset of such triples. We aim to prove that $I(\vec i) = \tilde I(\vec i)$.

\begin{propn}
\label{pseudo-parabolic-irreducible}
 For each $\vec i \in \Z_+^3$:
\begin{enumerate}
  \item  $\Char(M_{\vec i})=\Char(X_{\vec i})\Char(M)$ for all $q$,
  \item all $M_{\vec i}$ with $\vec i\in \Z_+^3$ are in $\O_t(\Hbb P^2)$.
 \end{enumerate}
\end{propn}
 \begin{proof}

 Consider $\tilde M_{\vec i}$ as a $U_q(\g_-)$-module, $\tilde M_{\vec i}\simeq U_q(\g_-)/\tilde I_{\vec i}^-$, which makes sense at $q=1$
 too\footnote{Although the action of $U_q(\g)$  on these modules does not extend to the classical point $q= 1$, they are
quasi-classical as modules over $U_q(\g_-)$ and equipped with an obvious grading by weights.}.
 It the classical limit
 $q\to 1$, it  goes to a quotient of $U(\g_-)$ by the left ideal generated by
 $f_{\bt_s}^{i_s+1}$, $s=1,2,3$.
Therefore
$$
\Char(M_{\vec i})\leqslant \Char (\tilde M_{\vec i})\leqslant \Char(X_{\vec i}) \Char(\g_-/\k_-)e^{\la}=\Char(X_{\vec i})\Char(M),
$$
at generic $q$. That is, the inequalities hold for dimensions of subspaces of the same weight for almost all $q$.
 The set of $q$-s where they are violated may depend on the weight.

Suppose  that $\Char(M_{\vec i})=\Char (X_{\vec i})\Char(M )$
for each  $M_{\vec i}\subset V^m\tp M$, $m\geqslant 0$, at all $q$. That holds trivially for $m=0$.
 The direct sum decomposition $V\tp M_{\vec i}\simeq\op_{\vec j\in I(\vec i)} M_{\vec j}$
implies
\be
\Char(V) \Char(M_{\vec i})&=&\sum_{\vec j\in I(\vec i)}\Char (M_{\vec j})
    \leqslant\sum_{\vec j\in \tilde I(\vec i)}\Char (\tilde M_{\vec j})\leqslant
   \nn\\
    &\leqslant& \sum_{\vec j\in \tilde I(\vec i)}\Char(X_{\vec j})\Char(M)=\Char(V)\Char(X_{\vec i})\Char(M)
\label{car-eq}
\ee
for generic $q$,
because $\op_{\vec j\in \tilde I(\vec i)} X_{\vec j}=V\tp X_{\vec i}$.
We conclude that the inequalities in (\ref{car-eq}) are all equalities (for generic $q$), and, secondly,
$I(\vec i)=\tilde I(\vec i)$. In particular,  for each $\vec j$ and  each weight $\mu$ we have
\be
\dim M_{\vec j}[\mu] =\dim \tilde M_{\vec j}[\mu]=\dim (X_{\vec j}\tp  M)[\mu]
\label{eq_char}
\ee
at all $q$ in a punctured neighbourhood of $1$ (that might depend on
$\vec i$ and $\mu$). Then
$\Char (M_{\vec j}) =\Char(X_{\vec j})\Char(M)$
at all $q$   as $M_{\vec j}$  is simultaneously a
quotient of a Verma module and is a submodule in $V^{\tp (m+1)}\tp M$,
which are both flat at all $q$ including $q=1$.
Induction on $m$ proves 1) for all    $M_{\vec j}$.

To prove 2), we use the equality $\tilde I(\vec i)= I(\vec i)$ we have already established. That is, for each weight $\eta$ of a singular vector in
the $\k$-module $V\tp X_{\vec i}$ the pseudo-parabolic module of highest weights $\la+\eta$ does appear
in $V\tp M_{\vec i}$ (uniquely since all weights in $V$ are multiplicity free). Again induction on $m$ such that $V^{\tp m}\tp M\supset M_{\vec i}$
along with Proposition \ref{all Verma in PPO} secures 2).
 \end{proof}
 \begin{corollary}
 \label{multiplicities}
  For every $V\in \Fin_q(\g)$ and for all $\vec i, \vec j\in \Z_+^3$,
  there is an  isomorphism
  $$\Hom_{U_q(\g)}(M_{\vec j},V\tp M_{\vec i})\simeq\Hom_\k(X_{\vec j},V\tp X_{\vec i}).$$
\end{corollary}
\begin{proof}
The equality $\Char (V\tp M_{\vec i})=\sum_{\vec j \in I}\Char (M_{\vec j})$, where the summation is over an irreducible decomposition of $V\tp M_{\vec i}$,
implies $\Char (V\tp X_{\vec i})=\sum_{\vec j \in I }^{}\Char (X_{\vec j})$, thanks to Proposition \ref{pseudo-parabolic-irreducible}. Therefore  the $\k$-module $\op_{\vec j\in I}X_{\vec j}$ is isomorphic to $V\tp X_{\vec i}$
and the assertion follows.
\end{proof}

Now we summarise the main result of the paper.
\begin{thm}
\label{main}
  \begin{enumerate}
    \item   $\O_t(\Hbb P^2)$ is semi-simple for all $q$.
    \item For all $q$, $\O_t(\Hbb P^2)$  is equivalent to the category $\mathrm{Fin}(\k)$.
    \item Simple objects in $\O_t(\Hbb P^2)$ are exactly pseudo-parabolic Verma modules, for almost all $q$.
  \end{enumerate}
\end{thm}
\begin{proof}

  The category $\O_t(\Hbb P^2)$ is clearly additive.
  To prove the first   statement, observe that a module $V$ from $\Fin_q(\g)$ can be realized as a submodule
  in a tensor power of $\C^{6}$. Then apply Propositions \ref{all Verma in PPO} and \ref{VMi-comp-red}.

  Equivalence $\O_t(\Hbb P^2)\sim \mathrm{Fin}(\k)$ as Abelian categories can be proved similarly to \cite{M4}, Proposition 3.8
  (cf. also Corollary \ref{multiplicities} above).

  We know from Propositions \ref{VMi-comp-red} and \ref{pseudo-parabolic-irreducible} that simple objects of $\O_t(\Hbb P^2)$
  are exactly $M_{\vec i}$, $\vec i \in \Z_+^3$.
  Let us prove that for $M_{\vec i}\simeq \tilde M_{\vec i}$ for all but a finite number of values of $q$.

  Indeed,  a module of highest weight is irreducible if and only if its contravariant form is non-degenerate or, alternatively,
it has no singular vectors. Weights of singular vectors may be only  in  the orbit of the highest weight under the shifted action of the Weyl group.
Let $\tilde W\subset \tilde M_{\vec i}$ and  $W\subset M_{\vec i}$ denote the sums of weight spaces whose weights are in that orbit.
It is sufficient to check non-degeneracy of the form only on  $\tilde W$.
Since $\tilde W$ is finite dimensional, there is an alternative: either the form is degenerate for all $q$ or
or it is not at some and therefore almost all $q$. From (\ref{eq_char}) we see that  $\tilde W\simeq W$   in an open
neighbourhood of $1$. Therefore the form is non-degenerate  on $\tilde W$ and hence on $\tilde M_{\vec i}$ for almost all $q$
as required.
 \end{proof}
Note that the set of exceptional $q$ where $M_{\vec i}\not \simeq \tilde M_{\vec i}$ may depend on a module.
We nevertheless conjecture that it is empty for all $\vec i$, as  is the case for the base module.

\section{The algebra $\C_q[\Hbb P^2]$ and Reflection Equation}
\label{SecRE}
In this section we give a more detailed description of the quantized polynomial ring $\Ac=\C_q[\Hbb P^2]$ and its one-dimensional representation.
This is a special case of a general construction, and the reader is referred  to   \cite{M2,M6}  for details.

Let $\pi$  be the representation homomorphisms of $U_q(\g)$ to $\End(V)$, $V\simeq \C^6$.
Pick up a basis $\{v_i\}_{i=1}^6\subset V$
 as in Section \ref{SecETEP}. Let $\nu_i$ denote the weight of $v_i$, then
  $\nu_{i}=-\nu_{i'}$, where $i'=7-i$. Denote $\varsigma_i=1$ and $\varsigma_{i'}=-1$ for $i=1,2,3$.

Let $\Ru$ be a universal  R-matrix of $U_q(\g)$ such that $(\pi\tp \id)(\Ru)\in \End(\C^6)\tp U_q(\b_+)$ and set $\Q=\Ru_{21}\Ru$.
It commutes with the coproduct of every element in $U_q(\g)$.
Denote by $P$ the  flip of the tensor factors in $\C^6\tp \C^6$ and fix a $U_q(\g)$-invariant
braid matrix  $S\in \End(\C^6)\tp \End(\C^6)$. Note that $R=PS$ needs not to be image of the particular $\Ru$ entering $\Q$:
e.g. one can take $R=(\pi\tp \pi)(\Ru_{21}^{-1})$. One can choose $\pi$ and $R$ as in \cite{FRT}.

It is known that $\C_q[G]$ can be realized as  the locally finite part of the adjoint $U_q(\g)$-module. It is a subalgebra in $U_q(\g)$ generated
by entries of the matrix $(\pi\tp \id)(\Q)$.
The  image of $\C_q[G]$ in $\End(M)$ is a flat deformation of a quotient of $\C[G]$ by the defining ideal of $\Hbb P^2$. That
is a maximal proper invariant ideal in $\C[G]$, whence the image is a quantization of $\C[\Hbb P^2]$,
see \cite{M2} for details.

Let $\varpi\propto \sum_{i,j=1}^{6} q^{\rho_i-\rho_j}\varsigma_i\varsigma_je_{i'j}e_{ij'}$ be the invariant projector onto the trivial one-dimensional submodule in $\C^6\tp \C^6$. Here   $\rho_i=(\rho,\nu_i)=-(\rho,\nu_{i'})$;  in particular,  $\rho_i={4-i}$ for $i=1,2,3$.

Let $\pi_{\vec i}$, $\vec i\in \Z_+^3$, denote the representation homomorphism $U_q(\g)\to \End(M_{\vec i})$. The operator  $(\pi\tp \pi_{\vec i})(\Q)$ has
eigenvalues
\be
\label{eigenvalues}
x_\nu=q^{2(\la+\xi+\rho,\nu)-2(\rho,\ve_1)}, \quad \nu\in \La(V^+_{\vec i}),
\ee
where $\la+\xi$ is the highest weight of $M_{\vec i}$.
In particular, the matrix $Q=(\pi\tp \pi_{\vec 0})(\Q)$ has
two eigenvalues
$
q^{2(\la+\rho,\ve_1)-2(\rho,\ve_1)}
$
and
$
q^{2(\la+\rho,\ve_3) -2(\rho,\ve_1) }
$
on $\C^6\tp M$
corresponding to irreducible submodules of highest weights $\ve_1+\la$ and $\ve_3+\la$.
The value of its q-trace $\Tr_q(Q)=\Tr\bigl(\pi(q^{2h_\rho})Q\bigr)$  on $M$ can be found by the formula $\Tr_q(Q)=\Tr\bigl(\pi(q^{2h_\rho+2h_\la})\bigr)$,
cf.  \cite{M1}.

The algebra  $\Ac$ is generated by the entries  $\{Q_{ij}\}_{i,j=1}^6$, which satisfy
$$
S_{12}Q_2 S_{12} Q_2 = Q_2 S_{12}Q_2 S_{12}, \quad Q_2 S_{12}Q_2\varpi_{12} =q^{-7}\varpi_{12} =\varpi_{12}Q_2 S_{12}Q_2,
$$
$$
(Q+q^{-2})(Q-q^{-4} )=0, \quad \Tr_q(Q)=-(q^4+q^{-4}).
$$
Equations of the first line are understood in $\End(\C^6)\tp \End(\C^6)\tp \End(M)$ and the subscripts label the $\End(\C^6)$-factors.
They are equations of $\C_q[G]$, a deformation of $\C[G]$ that is equivariant under the conjugation action of $G$ on itself.
The last two equations fix the quantized conjugacy class $\Hbb P^2$.
This is the full set of relations defining $\Ac$.

There is a one-dimensional representation $\chi\colon \Ac\to \C$,  $Q_{ij}\mapsto A_{ij}$, where
$$
A
=
-q^{-3}\left(
\begin{array}{cccccc}
 q-\bar q & 0 & 0 & 0 & 1 & 0 \\
 0 &q-\bar q  & 0 & 0 & 0 & -1 \\
 0 & 0 &    q & 0 & 0 & 0 \\
 0 & 0 &    0 & q & 0 & 0 \\
 1 & 0 &    0 & 0 & 0 & 0 \\
 0 & -1 &   0 & 0 & 0 & 0 \\
 \end{array}
\right).
$$
In the classical limit, the matrix $A$ goes over to a point $t'\in \Hbb P^2$ where the Poisson bracket vanishes.

The matrix $A$ defines an embedding  of $\Ac$ in  the restricted Hopf dual to $U_q(\g)$ that we denote by $\Tc$.
A description of the algebra $\Tc$ can be extracted from \cite{FRT}.
Let   $T=(T_{ij})_{i,j=1}^{6}$ denote its matrix of generators.
This matrix is invertible with $(T^{-1})_{ij}=\gm(T_{ij})$, where
$\gm$ is the antipode of $\Tc$.
One has two commuting left and right translation actions of $U_q(\g)$ on $\Tc$ expressed through the Hopf paring and the
comultiplication in $\Tc$ by
$$
h\tr a= a^{(1)}(h,a^{(2)}),\quad  a\tl h= (a^{(1)},h) a^{(2)},\quad  a\in \Tc, \quad h\in U_q(\g).
$$
They are compatible with multiplication on $\Tc$ making it a $U_q(\g)$-bimodule algebra.

The assignment $Q_{ij}\mapsto (T^{-1}A T)_{ij}$
defines an equivariant homomorphism  $\Ac\to \Tc$, where $\Tc$
is viewed as a $U_q(\g)$-module  under the left translation action. It is an embedding by similar deformation
arguments as with the case of $\Ac\subset \End(M)$.
The  character
$\chi$ factors through the composition $\Ac\to  \Tc\to \C$,
where the right arrow is the counit $\epsilon$.

The entries of the matrix
$$
\Kc=(\id\tp \pi)(\Ru_{12})A_2(\id\tp \pi)(\Ru_{21})\in   \U_q(\g)\tp \End(\C^{6})
$$ generate a left coideal subalgebra $U_q(\k')\subset \U_q(\g)$. It is a deformation of
$U(\k')$ with $\k'\simeq \k$ being the Lie algebra of the centralizer of $t'$.

One can check that  $a\tl b=\epsilon(b)a$ for all $b\in U_q(\k')$ and $a\in \Ac$.  We argue that $\Ac$ exhausts
all of the subalgebra of $U_q(\k')$-invariants, for generic $q$. Indeed, the latter is  $\cap_{i,j=1}^6 \ker \Kc'_{ij}$
where $\Kc'_{ij}=\Kc_{ij}-\eps(\Kc_{ij})\in \k' \mod (q-1)$. Restricted to every isotypic component  of
the Peter-Weyl decomposition of $\Tc$, the kernel cannot increase in deformation.

\section{Quantization of equivariant vector bundles on $\Hbb P^2$}

In this section, we will interpret $\O_t(\Hbb P^2)$
as a category of "representations" for  quantum vector bundles on $\Hbb P^2$.

In the classical algebraic geometry, global sections of vector bundles on a variety
are finitely generated projective modules over its coordinate ring. If
a group $G$ acts on the bundle coherently with the base, the vector
bundle is called equivariant.
Algebraically it means that $G$ acts on global functions by automorphisms, $G$ acts on global sections,
and the multiplication between functions and sections is equivariant.

In the case of homogeneous space $G/K$, a vector bundle $\Gamma(G/K,X)$ is characterized by
a finite dimensional $K$-module $X$ over the initial point. It can be realized
as the space of $K$-invariants in  $\C[G]\tp X$ under   right translations. The group $G$
acts on $\Gamma(G/K,X)\simeq (\C[G]\tp X)^K$ by left translations.

 For a reductive pair $G\supset K$, the Peter-Weyl decomposition
 $\C[G]=\sum_{[V]} V\tp V^*$ gives the isotypic component of an irreducible module $V$ in $\Gamma(G/K,X)$;
 it is $V\tp \Hom_K(X,V)$. This is the classical input that we are going to mimic in our approach to quantization.

We have already  argued that the base module $M$ supports a faithful representation of $\Ac$
as a subalgebra in the locally finite part $\End^\circ(M)$ of linear operators on $M$.
Similarly we claim that the locally finite part $\Hom^\circ(M,M_{\vec i})$ of the $U_q(\g)$-module of linear maps from $M$ to $M_{\vec i}$
is a quantization  of the vector bundle $\Gamma(\Hbb P^2,X_{\vec i})$ with fiber $X_{\vec i}$. Note that $\Hom^\circ(M,M_{\vec i})$ is a natural equivariant right
$\End^\circ(M)$-module via the composition of linear maps.
\begin{propn}
\label{deformation_bundls}
  As a $U_q(\g)$-module, $\Hom^\circ(M,M_{\vec i})$ is a deformation of $\Gamma(\Hbb P^2,X_{\vec i})$.
\end{propn}
\begin{proof}
Since $M$ and $M_{\vec i}$ are irreducible along with their dual modules of lowest weight, equivariant maps  from $V$ to $\Hom(M,M_{\vec i})$
are in bijection with equivariant maps from $\Hom(M^*_{\vec i},M^*)$ to $V^*$,
for every $V\in \Fin_q(\g)$.
We have a version of  Corollary \ref{multiplicities} for dual modules
and  we can write
$$
\Hom_{U_q(\g)}\bigl(\Hom(M^*_{\vec i},M^*),V^*\bigr)\simeq \Hom_{U_q(\g)}\bigl(M_{\vec i}^*,V^*\tp M^*\bigr)\simeq   \Hom_\k(X_{\vec i}^*,V^*)\simeq \Hom_\k(V,X_{\vec i}).
$$
The rightmost term is isomorphic to $\Hom_\k(X_{\vec i},V)$ as $V$ is completely reducible over $\k$.
Thus the isotypic component of $V$ in $\Hom^\circ(M,M_{\vec i})$ is a deformation of the isotypic component of its classical
counterpart in
$\Gamma(\Hbb P^2,X_{\vec i})$.
\end{proof}
In particular, setting $M_{\vec i}=M$ we conclude that $\End^\circ(M)$ has the same module structure as $\Ac$. This implies that, for $q\not =1$, the algebra $\Ac$ exhausts all
of $\End^\circ(M)$. We will give a recipe for   construction of $\Hom^\circ(M,M_{\vec i})$ in what follows.
 For each $V\in \Fin_q(\g)$ an invariant projector from $\End(V)\tp \End(M)$
is in   $\End(V)\tp \End^\circ(M)$ and therefore in $\End(V)\tp \Ac$.
Such projectors can be constructed with the help of the invariant element $\Q$.
 \begin{lemma}
\label{separation}
  For each $\vec i\in \Z_+^3$,  the operator $\Q$ separates irreducible components in $\C^6\tp M_{\vec i}$.
\end{lemma}
\begin{proof}
Let $\la+\xi$,  $\xi\in \La^+_\k$, be the highest weight of $M_{\vec i}$.
We will calculate the ratio of eigenvalues $x_\mu x_\nu^{-1}$ for  $\mu\not=\nu$ using the formula (\ref{eigenvalues}).
By definition of the base weight, we find
$$
x_\mu x_\nu^{-1}=q^{2(\la+\xi+\rho,\mu-\nu)}=
\al(t)q^{2(\kappa+\xi,\al)}, \quad \al=\mu-\nu\in \Rm_\g.
$$
Here we used the fact that all non-zero weight differences in $\C^6$ are roots.
The right-hand side cannot turn $1$ if $\al\in \Rm_\g\backslash \Rm_\k$, because it has the form $-q^\Z$, and $q$ is not a root of unity.
On the other hand, if  $\al\in \Rm_\k^+$, then
$
x_\mu x_\nu^{-1}= q^{2(\kappa+\xi,\al)}\not =1
$
either
because
$
(\kappa,\al) >0
$
and
$
(\xi,\al) \geqslant 0.
$
\end{proof}
It turns out that the matrix $Q$ together with intertwiners from $\Fin_q(\g)$ are enough to get all morphisms in $\O_t(\Hbb P^2)$.
The braid matrix $S$ from the previous section produces a family of $U_q(\g)$-invariant operators on the tensor algebra $T(V)$ of the module $V=\C^6$
in the standard way, see e.g. \cite{FRT}.
Denote by  $\mathcal{I}$ the algebra of invariant operators on $T(V)\tp M$ generated by the matrix $Q\in \End(V\tp M)$ and all invariant operators on $T(V)$.

\begin{propn}
   $\Ic$ exhausts all of the algebra of invariant operators on $T(V)\tp M$.
  \label{idempotents}
\end{propn}
\begin{proof}
We need to show that $\Ic$ separates
  submodules in $V^{\tp m}\tp M$ for all $m\geqslant 0$. We do it by induction on $m$.

  The assertion is true for $m=0$ because $M$ is irreducible.
  Suppose that is done for some $m\geqslant 1$ and pick up $M_{\vec i}\subset V^{\tp m}\tp M$ with the representation $\pi_{\vec i}\colon U_q(\g) \to \End(M_{\vec i })$.
  Choose an invariant projector  $P_{\vec i}\colon V^{\tp m}\tp M\to M_{\vec i}$. By induction assumption, $P_{\vec i}$ belongs to $\Ic$.

  Observe that the image of the operator $(\id \tp \Delta^ m)(\Q)$ in $\End(V^{\tp m})\tp \End(M)$ belongs to $\Ic$ for all $m$.
  This readily follows from the identity $(\id \tp \Delta)(\Q)=\Ru^{-1}_{12}\Q_{13}\Ru_{12} \Q_{23}$,
  which reduces $(\id \tp \Delta^m)(\Q)$ to a product of $S$-and $Q$-matrices, \cite{DKM}.
 Therefore the operator
$$
 \bigl(\pi^{\tp ( m+1)}\tp \pi_{\vec 0}\bigr)(\id \tp \Delta^{m+1})(\Q)\times (\id \tp P_{\vec i}) \in \Ic
$$
separates irreducible submodules in $V\tp M_{\vec i}$, by
Lemma \ref{separation}. This is true for each summand
in the decomposition $V\tp V^{\tp m}\tp M=\op_{\vec i} V\tp M_{\vec i}$.
Induction on $m$ is completed.
\end{proof}
By construction, $\Ic$ is a subalgebra in $T\bigl(\End(V)\bigr)\tp \Ac$. Applying $\chi$ to the right factor one obtains
a subalgebra $\Ic_A$ of $U_q(\k')$-invariant operators in $T(V)$. It is generated by the matrix $A$ over the subalgebra of $U_q(\g)$-invariant operators
on $T(V)$, cf. \cite{M6}, Proposition 4.5. It follows from Proposition \ref{idempotents} above that $\Ic_A$ is exactly the  commutant of $U_q(\k')$ in $T\bigl(\End(V)\bigr)$.

In the remaining part of the section we prove equivalence of  $\O_t(\Hbb P^2)$ and  the category $\mathrm{Pr}_q(\Ac,\g)$ of  equivariant finitely  generated projective $\Ac$-modules.
All morphisms in  $\mathrm{Pr}_q(\Ac,\g)$ are equivariant. Objects are direct summands in $\Ac$-modules freely generated by
$U_q(\g)$-modules from $\Fin_q(\g)$.
\begin{lemma}
\label{evaluation_hom}
  For every $N\in \O_t(\Hbb P^2)$, the evaluation map $\Hom^\circ(M,N)\tp M\to N$, $\phi\tp m\mapsto \phi(m)$
  factors through an isomorphism $\Hom^\circ(M,N)\tp_\Ac M\to N$.
\end{lemma}
\begin{proof}
  First suppose that $N\not =\{0\}$ is irreducible. As the map is equivariant, its image is a submodule in $M$ and hence coincides with $M$ because
  $\Hom^\circ(M,N)\not =\{0\}$, by Proposition \ref{deformation_bundls}.
  In general, $N$ is a direct sum of irreducibles, $N=\op_iN_i$. Then $\Hom^\circ(M,N)=\op_i\Hom^\circ(M,N_i)$, and
  the assertion follows.
\end{proof}
Every module $N$ from  $\O_t(\Hbb P^2)$ is a direct summand in $V\tp M$ for some $V\in  \Fin_q(\g)$, therefore
$\Hom^\circ(M,N)$ is a direct summand in a free equivariant $\Ac$-module $\Hom^\circ(M,V\tp M)$.
The assignment $N\mapsto \Hom^\circ(M,N)$
 is a covariant functor from $\O_t(\Hbb P^2)$ to the category $\mathrm{Pr}_q(\Ac,\g)$, which we denote by $\mathfrak{H}$.
 It is obviously additive and respects
 tensor multiplication by modules from  $\Fin_q(\g)$.

 \begin{propn}
 The  functor
 $\mathfrak{T}\colon \Gamma \mapsto \Gamma\tp _\Ac M$ is the left inverse to $\mathfrak{H}$.
 \end{propn}
\begin{proof}
  Let $N$ be a module from $\O_t(\Hbb P^2)$ and $P$ be an invariant projector $V\tp M\to N$ for some $V\in \Fin_q(\g)$.
  As we commented after Proposition \ref{deformation_bundls}, $P\in \End(V)\tp \Ac$.
  Then $\Hom^\circ(M,N)$ is isomorphic to $P(V\tp \Ac)$ and $P(V\tp \Ac)\tp_\Ac M=N$ because $\Ac M=M$.

  If $f\colon N_1\to N_2$ is a $U_q(\g)$-homomorphism and $\phi\in \Hom^\circ(M,N_1)$, then $\mathfrak{H}(f)(\phi)=f\circ \phi$ is a map from $\Hom^\circ(M,N_2)$.
We get $\bigl(\mathfrak{H}(f)(\phi)\bigl)(m)=(f\circ \phi)(m)=f\bigl(\phi(m)\bigr)$ for all $m\in M$.
  Applying Lemma \ref{evaluation_hom} we arrive at   $(\mathfrak{T}\circ\mathfrak{H})(f)=f$.
\end{proof}
 The functor $\mathfrak{H}$ is surjective on objects up to an isomorphism. If $V\in \Fin_q(\g)$ and
$P(V\tp \Ac)$ is an  $\Ac$-module from $\mathrm{Pr}_q(\Ac,\g)$ determined by an invariant projector $P\in \End(V)\tp \Ac$, then $P(V\tp \Ac)$ is  isomorphic to $\Hom^\circ(M,N)$ with  $N=P(V\tp M)\in \O_t(\Hbb P^2)$ because $\Ac\simeq \End^\circ(M)$.
\begin{thm}
The  $\Fin_q(\g)$-module categories $\O_t(\Hbb P^2)$  and $\mathrm{Pr}_q(\Ac,\g)$ are equivalent.
\end{thm}
\begin{proof}
 We have seen that $\mathfrak{H}$ is surjective on objects and injective on morphisms.
We are left to check that it is surjective on morphisms as well.

Suppose that $G\colon \Gamma_1\to \Gamma_2$ is a morphism in $\Pr_q(\Ac,\g)$.
We can assume that $\Gamma_i=\mathfrak{H}(N_i)$ for some  $N_i\in \O_t(\Hbb P^2)$, $i=1,2$.
Denote by $\jmath_i\colon N_i\to V_i\tp M$ and by $\wp_i\colon V_i\tp M\to N_i$ their embeddings and projections, respectively, such that $\wp_i\circ \jmath_i=\id_{N_i}$.
They give rise to embeddings and projections $\mathfrak{H}(\jmath_i)\colon \Gamma_i\to V_i\tp \Ac$ and
$\mathfrak{H}(\wp_i)\colon  V_i\tp \Ac\to \Gamma_i$, satisfying $\mathfrak{H}(\wp_i)\circ \mathfrak{H}(\imath_i)=\id_{\Gamma_i}$, $i=1,2$.

Consider  a morphism  $F=\mathfrak{H} (\jmath_2)\circ G\circ \mathfrak{H} (\wp_1)$ from $V_1\tp \Ac$ to $V_2\tp \Ac$. It implies that
 $G=\mathfrak{H} (\wp_2)\circ F\circ \mathfrak{H} (\jmath_1)$.
An equivariant map $F\colon (V_1\tp 1_\Ac)\to V_2\tp \Ac$  gives rise to an equivariant map $f\in V_1\tp M\to V_2\tp M$
because $\Ac\simeq \End^\circ(M)$. Then $F=\mathfrak{H}(f)$, and
 $G=\mathfrak{H} (\wp_2\circ f\circ  \jmath_1)$, hence $\mathfrak{H}$ is bijective on morphisms.
This completes the proof.
\end{proof}
Note that the category of general projective $\Ac$-modules is not semi-simple as a quotient of two projectives is not necessarily so.
The case of  equivariant projective modules is different.

The presence of a one-dimensional representation $\chi\colon \Ac\to \C$ from the previous section enables a realization of $\Pr_q(\Ac,\g)$ via quantized functions on
the group $G$. This construction is a deformation of the classical realization of an associated vector bundle.
Define $\Fin_q(\k')$ as the category of modules that are submodules of modules from $\Fin_q(\g)$. It is a $\Fin_q(\g)$-module category as $U_q(\k')$ is
a coideal subalgebra in $U_q(\g)$.

Given $X\in \Fin_q(\k')$ define the associated bundle with fiber $X$ as the subspace of $U_q(\k')$-invariants in $\Tc\tp X$.
It is in $\Pr_q(\Ac,\g)$ because for all $V\in \Fin_q(\g)$ there is a natural bijection between $U_q(\g)$-invariant idempotents in $\End(V)\tp \Ac$ and $U_q(\k')$-invariant projectors on $V$, cf. \cite{M6}.
The inverse functor acts by $\Gamma\mapsto \Gamma\tp _\Ac \C$ for $\Gamma \in \Pr_q(\Ac,\g)$.
This yields an equivalence between $\Fin_q(\k')$ and $\Pr_q(\Ac,\g)\sim\O_t(\Hbb P^2)$,
which obviously respects the action of $\Fin_q(\g)$.

\appendix

\section{Appendix}
In this technical section, we derive some identities in the algebra $U_q(\g_-)$ which
are needed for this exposition.

\begin{lemma}
  Define $\bar f_\theta$ obtained from $f_\theta$ by replacement $q\to \bar q$. Then
\be
f_\theta&=&[f_2,[[f_{1},f_2]_{\bar q},f_{3}]_{\bar q^2}]_q=\bar q[[f_{1},f_2]_{\bar q},[f_2,f_{3}]_{q^2}]_{\bar q},
\label{Ap-f-theta}
\\
\Bar f_\theta&=&[f_2,[[f_{1},f_2]_{q},f_{3}]_{q^2}]_{\bar q}=q[[f_{1},f_2]_{q},[f_2,f_{3}]_{\bar q^2}]_{q}.
\label{Ap-bar-f-theta}
\ee
\end{lemma}
\begin{proof}
We will use a modified Jacobi identity
\be
[x,[y,z]_a]_b=[[x,y]_c,z]_{\frac{ab}{c}}+c[y,[x,z]_{\frac{b}{c}}]_{\frac{a}{c}},
\label{Jacobi}
\ee
which holds true for any elements $x,y,z$ of an associative algebra and any scalars $a,b,c$ with invertible $c$.
This can be verified by a direct calculation.

Now let us prove the right equality in  (\ref{Ap-f-theta}). Apply (\ref{Jacobi}) to $[f_2,[[f_{1},f_2]_{\bar q},f_{3}]_{\bar q^2}]_q$ choosing $c=\bar q$:
\be
[f_2,[[f_{1},f_2]_{\bar q},f_{3}]_{\bar q^2}]_q=[[f_2,[f_{1},f_2]_{\bar q}]_{\bar q},f_{3}]+\bar q[[f_{1},f_2]_{\bar q},[f_2,f_{3}]_{ q^2}]_{\bar q}.
\nn
\ee
The first summand vanishes thanks to the Serre relation of weight $-(2\al_2+\al_{1})$ whence (\ref{Ap-f-theta}) follows.
Then (\ref{Ap-bar-f-theta}) follows from (\ref{Ap-f-theta}) by replacement $q\to \bar q$.
\end{proof}

\begin{lemma}
One has
\be
qf_\theta+\bar q\bar f_\theta=[f_1,f_\dt]\in J,
\label{f_theta-bar_f_theta}
\ee
\end{lemma}
\begin{proof}
Apply  (\ref{Jacobi}) to $[f_1,f_\dt]=[f_1,[f_2,[f_2,f_3]_{q^2}]_{\bar q^2}]$  choosing $c=\bar q$. Then
$$
[f_1,f_\dt]=[[f_1,f_2]_{\bar q},[f_2,f_3]_{q^2}]_{\bar q}+\bar q[f_2,[f_1,[f_2,f_3]_{q^2}]_{q}]_{\bar q}.
$$
The first summand is $qf_\theta$ from (\ref{Ap-f-theta}). In the second summand, replace $[f_1,[f_2,f_3]_{q^2}]_{q}$ with $[[f_1,f_2]_{q},f_3]_{q^2}$,
then it becomes $\bar q \bar f_\theta$ from (\ref{f_theta-bar_f_theta}).
\end{proof}
Other identities of interest can be also derived from the Serre relations a with the use of the modified Jacobi
identity (\ref{Jacobi}).
We will give another proof based on Lusztig's braid group automorphisms of $U_q(\g)$, \cite{ChP}.
\begin{propn}
  The following relations hold true in $U_q(\g_-)$:
\be
&[f_3,f_\theta]=0=[f_3,\Bar f_\theta],
\label{Ap-theta-beta}\\
&f_2 f_\theta=\bar q f_\theta f_2, \quad f_2\Bar f_\theta=q\Bar f_\theta f_2,
  \label{Ap-Serre_step}
\\
&f_\dt f_\theta= \bar q^2 f_\theta f_\dt,
\label{Ap-theta delta}
\\
&    f_\nu f_\theta=qf_\theta f_\nu,
\label{Ap-nu-theta}
\\
&f_\xi f_\theta =qf_\theta f_\xi,
\label {Ap-xi-theta}
\ee
where $f_\nu=[f_1,f_2]_{\bar q}$.
\end{propn}
\begin{proof}
Let  $T_i$ be  Lusztig automorphisms  of $U_q(\g)$ corresponding to
  simple reflections $\si_i\colon \Rm\to \Rm$ relative the simple roots $\al_i$, as in \cite{ChP}.
  They satisfy braid group relations, of which we will need only
$$
T_2 T_3 T_2 T_3=T_3 T_2 T_3 T_2.
$$
In particular, $f_\nu=T^{-1}_2(f_1)$ and $T^{-1}_3(f_2)=[f_2,f_3]_{\bar q^2}$ which implies
$$
T^{-1}_3T^{-1}_2T^{-1}_3(f_1)=T^{-1}_3(f_\nu)=f_\xi,$$ because $T^{-1}_3(f_1)=f_1$.
  Set $w=T^{-1}_3T^{-1}_2T^{-1}_3$, then
$$
w(f_1)\propto f_\xi,
\quad
w(f_2)=f_2,
\quad
w T^{-1}_2(f_1)\propto w(f_\nu)\propto f_\theta,
\quad
wT^{-1}_2(e_3)\propto q^{-h_3}f_3.
$$
The first equality has been checked. The second equality is fulfilled because $\si_3\si_2\si_3(\al_2)=\al_2$. The third formula follows from the first two
as $w$ is an algebra automorphism. The  last one readily follows from the equality $T_2^{-1} T_3^{-1} T_2^{-1}(e_3)=e_3$
as a result of $T^{-1}_3(e_3)$, cf.\cite{ChP}.

Applying $wT^{-1}_2$ to a commuting pair $(e_3,f_1)$ one gets the left equality (\ref{Ap-theta-beta}) because $(\theta,\al_3)=0$.
Applying $w$ to a quasi-commuting pair $(f_2,f_\nu)$, one gets the left equality in (\ref{Ap-Serre_step}). The right
equalities in  (\ref{Ap-theta-beta}) and (\ref{Ap-Serre_step}) result from  replacement $q\to q^{-1}$. Then
(\ref{Ap-Serre_step}) follows since $f_\dt$ comprises two $f_2$-factors and one $f_3$-factor.
To prove (\ref{Ap-nu-theta}), apply $T^{-1}_2T^{-1}_3$ to a quasi-commuting pair of $f_1$
and $f_\nu\simeq T^{-1}_2 T^{-1}_3(f_1)$, using the equality $T^{-1}_3 (f_1)=f_1$ and the braid relation.
The formula (\ref{Ap-xi-theta}) is obtained by applying $w$ to quasi-commuting $f_2$ and $f_\nu$.
\end{proof}

\vspace{10pt}
\noindent
\underline{\large \bf Acknowledgement.}

\vspace{10pt}
\noindent
The second author (AM) is grateful to Steklov Mathematical Institute, St.-Petersburg Department,
where the revised version of the original manuscript was done,
for partial support of this work.
This  research   was  supported by a grant for creation
and development of International Mathematical Centers, agreement no.
075-15-2019-1620 of November 8, 2019, between Ministry of Science and
Higher  Education of Russia and PDMI RAS.

The authors are indebted  to the anonymous referee  whose careful reading of the manuscript and
valuable comments and suggestions greatly helped us to improve and extend the original text.

Data sharing not applicable to this article as no datasets were generated or analysed during the current study.
\end{document}